\documentclass[final]{publmathdeb}

\usepackage{amsmath, amsfonts, amssymb}
\usepackage{graphicx}
\usepackage[percent]{overpic}
\usepackage{tabularx}
\usepackage{algorithm,algorithmic}
\usepackage{colonequals}
\usepackage{epstopdf}

\newenvironment{ntabbing}{\setlength{\topsep}{0pt} \setlength{\partopsep}{0pt} \tabbing}{\endtabbing}

\newtheorem{theorem}{Theorem}[section]

\newtheorem{proposition}[theorem]{Proposition}
\newtheorem{corollary}[theorem]{Corollary}
\newtheorem{conjecture}[theorem]{Conjecture}
\theoremstyle{definition}

\theoremstyle{remark}

\numberwithin{equation}{section}

\newcommand{\ex}{\exists\:}
\newcommand{\fa}{\forall\:}
\newcommand{\ce}{\:\colonequals\:}

\newcommand{\abs}[1]{\left \lvert #1 \right \rvert}

\newcommand{\floor}[1]{\left \lfloor #1 \right \rfloor}
\newcommand{\ceiling}[1]{\left \lceil #1 \right \rceil}
\newcommand{\fractionalpart}[1]{\left \{ #1 \right \}}

\newcommand{\Max}{\operatorname{max}}

\newcommand{\Int}[1]{\operatorname{int}\left( #1 \right)}

\newcommand{\id}{\operatorname{id}}

\newcommand{\set}[1]{\left\{ #1 \right\}}
\newcommand{\I}{\mathrm{i}}
\newcommand{\N}{\mathbb{N}}
\newcommand{\Z}{\mathbb{Z}}

\newcommand{\R}{\mathbb{R}}
\newcommand{\C}{\mathbb{C}}
\newcommand{\Zi}{\mathbb{Z}\left[\mathrm{i}\right]}

\newcommand{\D}[1]{\mathcal{D}_{#1}}

\newcommand{\G}[1]{\mathcal{G}_{#1}}
\newcommand{\GN}[1]{\mathcal{G}_{#1}^{(0)}}
\newcommand{\GC}{\mathcal{G}_C}
\renewcommand{\vector}[1]{{\mathbf{#1}}}
\newcommand{\powerset}{\mathcal{P}}

\newcommand{\Line}{\,---$\mspace{4.511111111111mu}$\,}
\newcommand{\Dots}{$\mspace{1.4mu}\cdots$\,}
\newcommand{\MLine}{\mspace{5.411111111111mu}\text{---}\mspace{5.411111111111mu}}
\newcommand{\MDots}{\cdots}

\begin{document}

\title{On the characterization of Peth\H{o}'s Loudspeaker}

\author{Mario Weitzer}
\address{
Chair of Mathematics and Statistics\\
Montanuniversit\"at Leoben\\
Franz Josef-Stra\ss{}e 18, A-8700, Leoben\\
Austria
}
\email{mario.weitzer@unileoben.ac.at} 



\submitted{October 01, 2014}


\thanks{The author is supported by the Austrian Science Fund (FWF): W1230, Doctoral Program ``Discrete Mathematics''. The author would like to thank Prof. Peter Kirschenhofer and Prof. Attila Peth\H{o} for several stimulating discussions on the subject of this paper.}

\begin{abstract}
For $d \in \N$ and $\vector{r} \in \C^d$ let $\gamma_\vector{r}: \Zi^d \to \Zi^d$, where $\gamma_\vector{r}(\vector{a})=(a_2,\ldots,a_d,$ $-\lfloor\vector{r}\vector{a}\rfloor)$ for $\vector{a}=(a_1,\ldots,a_d)$, denote the {\em (d-dimensional) Gaussian shift radix system associated with $\vector{r}$}. $\gamma_\vector{r}$ is said to have the {\em finiteness property} iff all orbits of $\gamma_\vector{r}$ end up in $(0,\ldots,0)$; the set of all corresponding $\vector{r} \in \C^d$ is denoted by $\GN{d}$. It has a very complicated structure even for $d=1$.

In the present paper a conjecture on the full characterization of $\GN{1}$ - which is known as Peth\H{o}'s Loudspeaker - is formulated and proven in substantial parts. It is shown that $\GN{1}$ is contained in a conjectured characterizing set $\GC$. The other inclusion is settled algorithmically for large regions leaving only small areas of uncertainty. Furthermore the circumference and area of the Loudspeaker are computed under the assumption that the conjecture holds. The proven parts of the conjecture also allow to fully identify all so-called critical and weakly critical points of $\GN{1}$.
\end{abstract}

\maketitle

\section{Introduction}
\label{SIntroduction}
In 2005 Akiyama {\em et al.}~\cite{ABBPTI} introduced so-called {\em shift radix systems} (cf. also \cite{ABPTII, ABPTIII, ABPTIV}). For a natural number $d$ and a real vector $\vector{r} \in \R^d$ the mapping $\tau_\vector{r}: \Z^d \to \Z^d$ defined by
\begin{equation}
\tau_\vector{r}(\vector{a})=(a_2,\dots,a_d,-\floor{\vector{r} \vector{a}}) \qquad (\vector{a}=(a_1,\ldots,a_d)),
\end{equation}
is called the {\em d-dimensional shift radix system associated with $\vector{r}$} ({\em SRS}) and $\vector{r}$ its {\em parameter}. In \cite{BrunotteKirschenhoferThuswaldner11} the notion has been generalized to the complex setting. For a complex vector $\vector{r} \in \C^d$ the analogously defined mapping $\gamma_\vector{r}: \Zi^d \to \Zi^d$ is called the {\em d-dimensional Gaussian shift radix system associated with $\vector{r}$} ({\em GSRS}) (note that $\vector{r} \vector{a}\ce r_1a_1+\cdots+r_da_d$ and $\floor{z}\ce\floor{\Re(z)}+\I\floor{\Im(z)}$ for $z\in\C$). Let\footnote{$\N\ce\set{n\in\Z\mid n>0}.$}
\begin{align}
\G{d} &\ce \set{\vector{r} \in \C^d \mid \fa \vector{a} \in \Zi^d: \ex (m,n) \in \N^2: m \neq n \land \gamma_\vector{r}^m(\vector{a}) = \gamma_\vector{r}^n(\vector{a})}\\
\GN{d} &\ce \set{\vector{r} \in \C^d \mid \fa \vector{a} \in \Zi^d: \ex n \in \N: \gamma_\vector{r}^n(\vector{a}) = \vector{0}}
\end{align}
where for any $n \in \N_0$, $\gamma_\vector{r}^n(\vector{a})$ means the $n$-fold iterative application of $\gamma_\vector{r}$ to $\vector{a}$. The GSRS $\gamma_\vector{r}$ is said\footnote{\label{FN_identify} From now on a real vector $\vector{r}$ and its associated GSRS $\gamma_\vector{r}$ shall be identified in terms of properties.} to have the {\em finiteness property} iff $\vector{r} \in \GN{d}$.

SRS are closely related to two important notions of numeration systems. Indeed, as pointed out in \cite{ABBPTI, Hollander:96}, SRS form a generalization of $\beta$-expansions (see \cite{FS:92, Parry:60, Renyi:57}) and canonical number systems (CNS) (see \cite{KS:75, KP91, Pethoe:91} and \cite[Section~4.1]{Knuth:98}). The finiteness properties in the contexts of $\beta$-expansions and CNS are in one-to-one correspondence with the finiteness property for SRS.

GSRS on the other hand are a generalization of Gaussian numeration systems \cite{JR:95}. For a $\beta\in\Zi\setminus\set{0}$ and $\mathcal{C}\ce\set{c\in\Z[i]\mid\floor{c/\beta}=0}$ the pair $\left(\beta,\mathcal{C}\right)$ is called {\em Gaussian numeration system} iff every $x\in\Zi$ can be written uniquely in the form $x=a_0+a_1\beta+\ldots+a_n\beta^n$ where $n\in\N_0$, $a_i\in\mathcal{C}$ for $i\in\set{0,\ldots,n}$ and $c_n=0$ iff $x=0$. It is shown in \cite{BrunotteKirschenhoferThuswaldner11} that $(\beta,\mathcal{C})$ is a Gaussian numeration system iff $-1/\beta\in\GN{1}$. Furthermore the digit representation of $x$ with respect to $(\beta,\mathcal{C})$ is given by $a_i=\beta\fractionalpart{-1/\beta\gamma_{-1/\beta}^i(-x)}$ where $i\in\N_0$ and the fractional part of some $z\in\C$ is defined as $\fractionalpart{z}\ce\fractionalpart{\Re(z)}+\I\fractionalpart{\Im(z)}$.

In the present paper a conjecture on the characterization of $\GN{1}$ is given. Because of its shape and in honor of Attila Peth\H{o} $\GN{1}$ is known as Peth\H{o}'s Loudspeaker \cite{BrunotteKirschenhoferThuswaldner11}. In Section~\ref{SConjecture} the set $\GC$ is defined and it is conjectured that $\GN{1}=\GC$. The main result is stated in Section~\ref{SResult} and proven in Section~\ref{SInclusion1} and Section~\ref{SInclusion2} where it is shown that $\GN{1}\subseteq\GC$ and the other inclusion is settled in large parts leaving only a very small area of uncertainty. In Section~\ref{SConsequences} consequences like the Loudspeaker's circumference and area are derived under the assumption that the main conjecture holds. Finally all weakly critical and critical points of the Loudspeaker are identified in Section~\ref{SCritical}.

\section{The conjecture}
\label{SConjecture}
\begin{ntabbing}
Let \= $P_{10}(n)$ \= $\ce \left(1-\frac{1}{-n^2+30n+32},\frac{n+1}{-2n^2+60n+64}\right)$, \quad \= $n \in \Z$\kill
Let \> $P_0(1)$ \> $\ce (1,0)$, \quad $P_0(2) \ce \left(\frac{22}{\phantom{+}\mspace{-14mu}23},\frac{4}{23}\right)$, \quad $P_0(3) \ce \left(\frac{26}{\phantom{+}\mspace{-14mu}27},\frac{4}{27}\right)$\\
\> $P_1(n)$ \> $\ce \left(1-\frac{2}{n^2-2},\frac{n}{n^2-2}\right)$, \> $n \in \Z$\\
\> $P_2(n)$ \> $\ce \left(1-\frac{1}{n^2-n-1},\frac{n-1}{n^2-n-1}\right)$, \> $n \in \Z$\\
\> $P_3(n)$ \> $\ce \left(1-\frac{1}{n^2-n},\frac{n-1}{n^2-n}\right)$, \> $n \in \Z \setminus \{0, 1\}$\\
\> $P_4(n)$ \> $\ce \left(1-\frac{1}{\phantom{+}\mspace{-14mu}n^2},\frac{n}{n^2}\right)$, \> $n \in \Z \setminus \{0\}$\\
\> $P_5(n)$ \> $\ce \left(1-\frac{1}{n^2+1},\frac{n}{n^2+1}\right)$, \> $n \in \Z$\\
\> $P_6(n)$ \> $\ce \left(1-\frac{1}{n^2+n+1},\frac{n+1}{n^2+n+1}\right)$, \> $n \in \Z$\\
\> $P_7(n)$ \> $\ce \left(1-\frac{1}{n^2+n+2},\frac{n+1}{n^2+n+2}\right)$, \> $n \in \Z$\\
\> $P_8(n)$ \> $\ce \left(1-\frac{1}{n^2+2},\frac{n}{n^2+2}\right)$, \> $n \in \Z$\\
\> $P_9(n)$ \> $\ce \left(1-\frac{1}{n^2+3},\frac{n}{n^2+3}\right)$, \> $n \in \Z$\\
\> $P_{10}(n)$ \> $\ce \left(1-\frac{2}{n^2+n+6},\frac{n+1}{n^2+n+6}\right)$, \> $n \in \Z$
\end{ntabbing}
\bigskip
and let $\GC$ denote the union of the region bounded by the following infinite polygonal chain and the same region reflected at the real axis. The boundary of $\GC$ shall also be as given below where a solid line between two points indicates belonging of the corresponding line segment and an overline over a point indicates belonging of the corresponding vertex to $\GC$.\hfill
\begin{figure}[H]
\fontsize{7pt}{6pt}\selectfont
\begin{ntabbing}
$\overline{P_0(3)}$ \Line $\overline{P_2(7)}$ \Line \=$\overline{P_3(7)}$ \Dots \=$P_4(7)$ \Dots \=$\overline{P_5(7)}$ \=\Line $P_6(7)$ \Dots $P_7(7)$ \Dots $\overline{P_8(7)}$ \Line $\overline{P_9(7)}$\kill
\>\>$P_0(1)$ \Line $\overline{P_5(0)}$ \Line $P_6(0)$ \Dots\\
\>\>\>$\overline{P_5(1)}$ \Line $P_6(1)$ \Line $\overline{P_7(0)}$ \Dots $P_7(1)$ \Dots\\
\>\>\>$\overline{P_5(2)}$ \Line $\overline{P_6(2)}$ \Dots $P_7(2)$ \Dots $\overline{P_8(2)}$ \Dots\\
\>\>$P_4(3)$ \Dots $\overline{P_5(3)}$ \Line $P_6(3)$ \Dots $P_7(3)$ \Dots $\overline{P_8(3)}$ \Line\\
\>$\overline{P_3(4)}$ \Dots $P_4(4)$ \Dots $\overline{P_5(4)}$ \Line $P_6(4)$ \Dots $P_7(4)$ \Dots $\overline{P_8(4)}$ \Line\\
\>$\overline{P_3(5)}$ \Dots $P_4(5)$ \Dots $\overline{P_5(5)}$ \Line $P_6(5)$ \Dots $P_7(5)$ \Dots $\overline{P_8(5)}$ \Line $\overline{P_9(5)}$ \Line\\
$\overline{P_0(2)}$ \Line $\overline{P_2(6)}$ \Line $\overline{P_3(6)}$ \Dots $P_4(6)$ \Dots $\overline{P_5(6)}$ \Line $P_6(6)$ \Dots $P_7(6)$ \Dots $\overline{P_8(6)}$ \Line $\overline{P_9(6)}$ \Line\\
$\overline{P_0(3)}$ \Line $\overline{P_2(7)}$ \Line $\overline{P_3(7)}$ \Dots $P_4(7)$ \Dots $\overline{P_5(7)}$ \Line $P_6(7)$ \Dots $P_7(7)$ \Dots $\overline{P_8(7)}$ \Line $\overline{P_9(7)}$ \Line\\
$\overline{P_1(8)}$ \Line $\overline{P_2(8)}$ \Line $\overline{P_3(8)}$ \Dots $P_4(8)$ \Dots $\overline{P_5(8)}$ \Line $P_6(8)$ \Dots $P_7(8)$ \Dots $\overline{P_8(8)}$ \Line $\overline{P_9(8)}$ \Line $\overline{P_{10}(8)}$ \Dots\\
\>\>\>\> $\,\mspace{6.38mu}\vdots$\\
$\overline{P_1(\mspace{-0.9025mu}n\mspace{-0.9025mu})}$ \Line $\overline{P_2(\mspace{-0.9025mu}n\mspace{-0.9025mu})}$ \Line $\overline{P_3(\mspace{-0.9025mu}n\mspace{-0.9025mu})}$ \Dots $P_4(\mspace{-0.9025mu}n\mspace{-0.9025mu})$ \Dots $\overline{P_5(\mspace{-0.9025mu}n\mspace{-0.9025mu})}$ \Line $P_6(\mspace{-0.9025mu}n\mspace{-0.9025mu})$ \Dots $P_7(\mspace{-0.9025mu}n\mspace{-0.9025mu})$ \Dots $\overline{P_8(\mspace{-0.9025mu}n\mspace{-0.9025mu})}$ \Line $\overline{P_9(\mspace{-0.9025mu}n\mspace{-0.9025mu})}$ \Line $\overline{P_{10}(\mspace{-0.9025mu}n\mspace{-0.9025mu})}$ \Dots\\
\>\>\>\> $\,\mspace{6.38mu}\vdots$
\end{ntabbing}
\end{figure}
\vspace{-1.25\baselineskip}
\begin{conjecture}
\label{CMain}
If $\GC$ is as defined above then $\GN{1}=\GC$.
\end{conjecture}
Note that for all $i\in\{1, \ldots, 10\}: \underset{n \to \infty}{lim} P_i(n) = P_0(1)$. The following figure shows the part of $\GC$ which lies in the first quadrant and a magnification of the part where it gets regular. It can be seen that ultimately the boundary of $\GC$ consists of a sequence of pikes which have ten vertices each. For $n \in \N$ pike $n$ shall refer to the pike which contains the vertex $P_5(n)$.
\begin{figure}[H]
\includegraphics[width=5cm]{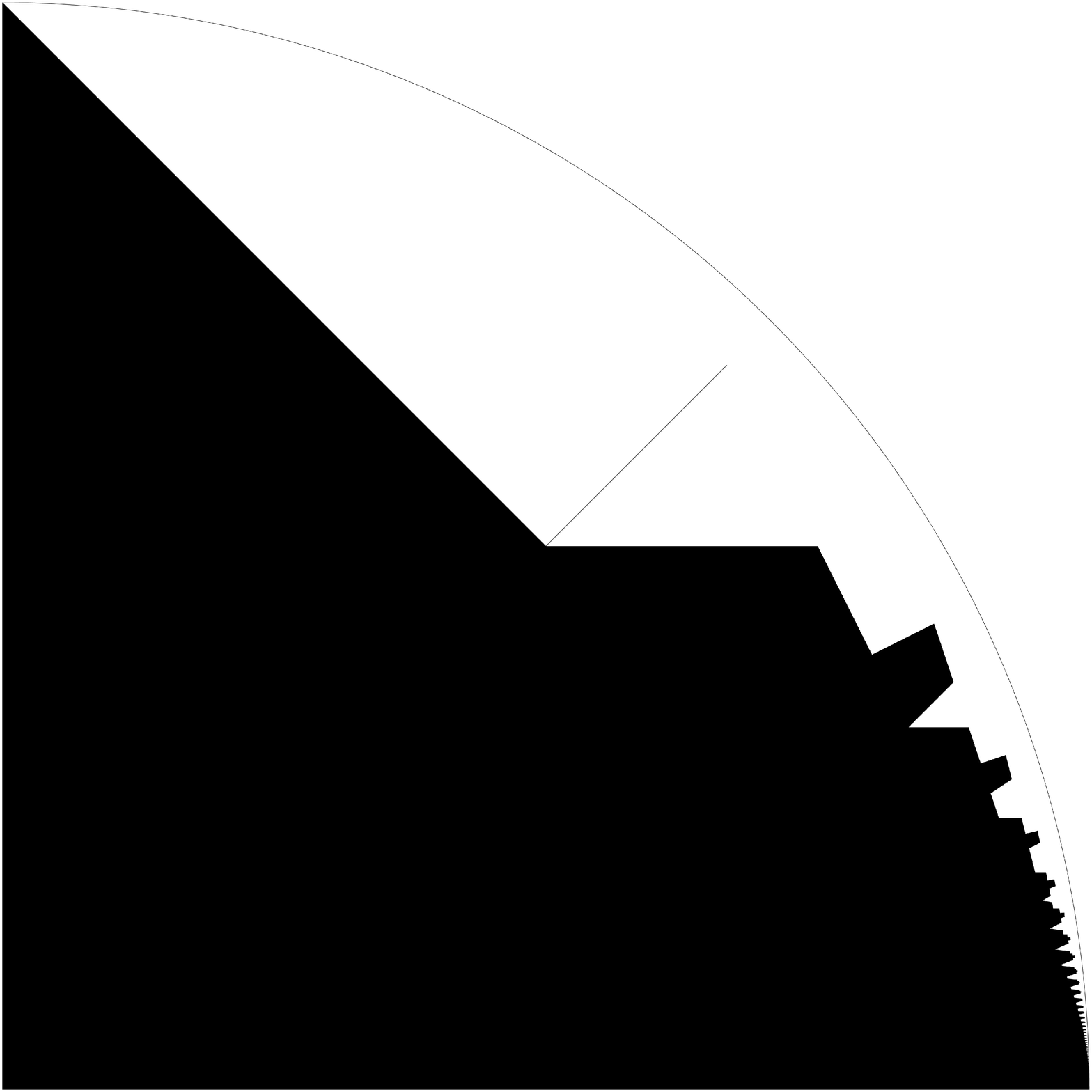}
\includegraphics[width=5cm]{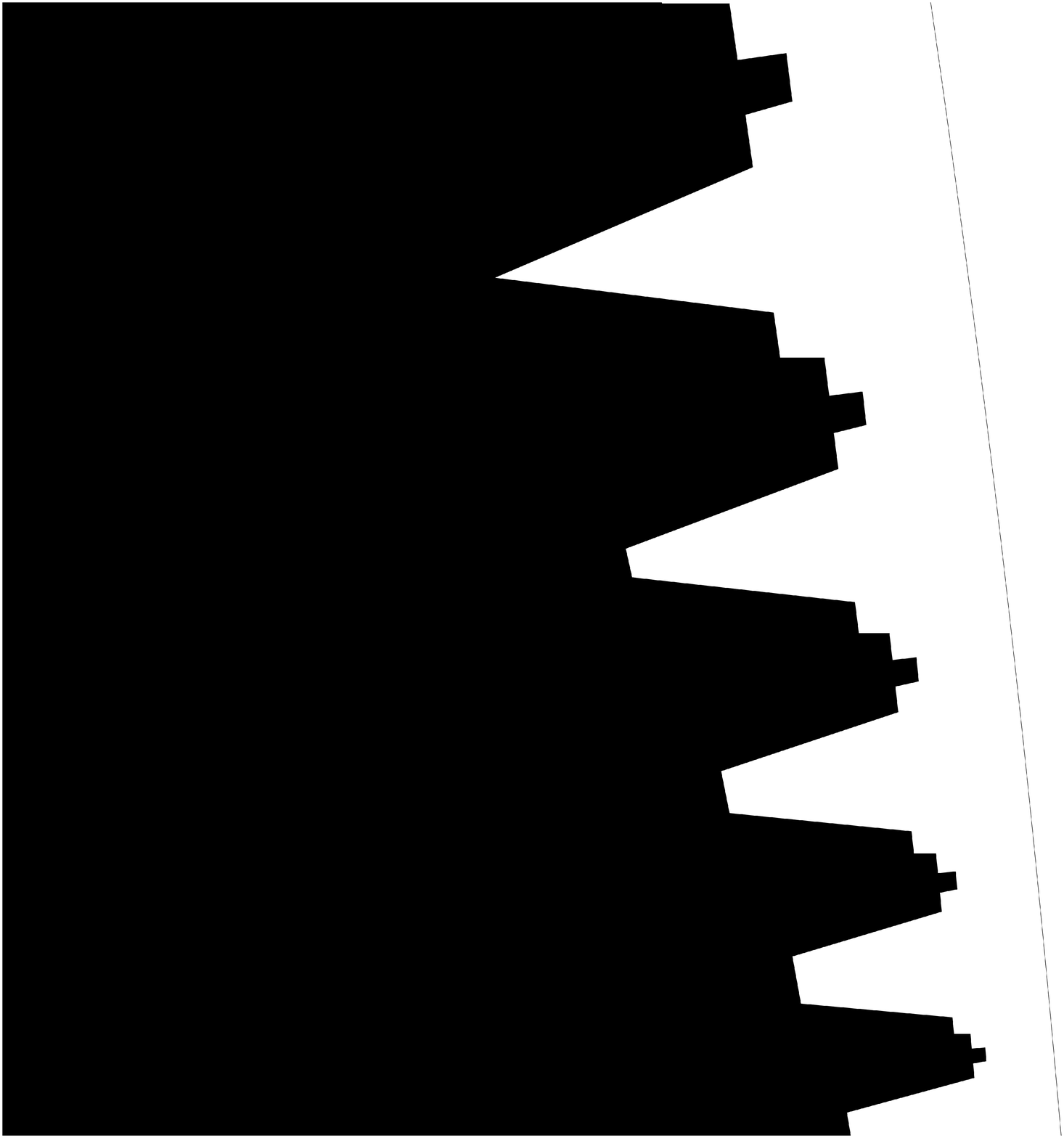}
\end{figure}

\section{The main result}
\label{SResult}
\begin{theorem}
\label{TMain}
Let $\GC$ be as in Section~\ref{SConjecture} and $D\ce\set{z\in\C\mid\abs{z}\leq\frac{2047}{2048}}$. Then
\begin{itemize}
\item[(i)]
$\GN{1} \subseteq \GC$
\item[(ii)]
$\GC \cap D \subseteq \GN{1} \cap D$
\end{itemize}
\end{theorem}
The first part is proven in Section~\ref{SInclusion1} and the second part in Section~\ref{SInclusion2}.

\section{Preliminaries}
\label{SPreliminaries}
For $n \in \N$, $\pi = (\vector{a}_1, \dots \vector{a}_n) \in (\Zi^d)^n$ is called a {\em cycle of $\vector{r}\in\C^d$} (or $\gamma_\vector{r}$, see footnote~\ref{FN_identify}) iff for all $i \in \set{1, \ldots, n}$ it holds that $\gamma_\vector{r}(\vector{a}_i) = \vector{a}_{i \bmod n + 1}$ (note that $\bmod$ has precedence over $+$ and $-$), a {\em cycle} iff there is a vector $\vector{r} \in \C^d$ for which $\pi$ is a cycle of $\vector{r}$, and {\em nontrivial} iff $\pi \neq (\vector{0})$, the {\em trivial cycle}. Let $P(\pi) \ce \set{\vector{r} \in \C^d \mid \pi \text{ cycle of } \vector{r}}$, the {\em associated polyhedron of $\pi$} or - if $\pi$ is a nontrivial cycle - the {\em cutout polyhedron of $\pi$}. $P(\pi)$ is either empty or the intersection of finitely many half spaces and therefore it does in fact always form a - possibly degenerate - convex polyhedron \cite{BrunotteKirschenhoferThuswaldner11}. It is clear that
\begin{equation}
\GN{d} = \G{d} \setminus \bigcup_{\pi \neq (\vector{0})} P(\pi)
\end{equation}
which provides a method to ``cut out'' regions (the cutout polyhedra) from $\G{d}$ \cite{ABBPTI}.

Cutout polyhedra can be used to prove that a given parameter $\vector{r}\in\C^d$ does not belong to $\GN{d}$. However they are insufficient to prove that it does belong. For that Brunotte's algorithm can be used (\cite[Theorem~5.1]{ABBPTI}) which is based on sets of witnesses. A set $V \subseteq \Zi^d$ is called a {\em set of witnesses for $\vector{r}$} iff it is stable under $\gamma_\vector{r}^{(1)}\ce\gamma_\vector{r}$, $\gamma_\vector{r}^{(2)}\ce-\gamma_\vector{r} \circ (-\id)$, $\gamma_\vector{r}^{(3)}\ce\operatorname{conj}\circ\gamma_\vector{r} \circ \operatorname{conj}\circ\id$, and $\gamma_\vector{r}^{(4)}\ce-\operatorname{conj}\circ\gamma_\vector{r} \circ (-\operatorname{conj}\circ\id)$ (where $\id$ is the identity on $\Zi$ and $\operatorname{conj}$ is the function on $\C^d$ which replaces every entry of the input vector by its complex conjugate) and contains a generating set of the group $(\Zi^d, +)$ which is closed under taking inverses. Every such set of witnesses has the decisive property
\begin{equation}
\vector{r} \in \GN{d} \Leftrightarrow \fa \vector{a} \in V: \ex n \in \N: \gamma_\vector{r}^n(\vector{a}) = \vector{0}.
\end{equation}
In the case of a finite set of witnesses this provides a method to decide whether or not a given parameter $\vector{r}$ belongs to $\GN{d}$. This is what Brunotte's algorithm does for any given parameter $\vector{r}$ in the interior of $\G{d}$ - it finds a finite set of witnesses. It shall be denoted by $V_\vector{r}$ - the {\em set of witnesses associated with $\vector{r}$} - and can be computed using the following iteration:
\begin{equation}
\begin{split}
V_0&\ce \{(\pm 1,0,\ldots,0),\ldots,(0,\ldots,0,\pm 1),\\
&\phantom{\vphantom{a}\ce\vphantom{a}\{}(\pm \I,0,\ldots,0),\ldots,(0,\ldots,0,\pm \I)\}\\
\fa n\in\N: V_n&\ce V_{n-1} \cup \gamma_\vector{r}^{(1)}(V_{n-1}) \cup \ldots \cup \gamma_\vector{r}^{(4)}(V_{n-1})\\
V_\vector{r}&\ce \displaystyle \bigcup_{n \in \N_0} V_n
\end{split}
\end{equation}
If $\vector{r}$ is an element of the interior of $\G{d}$ the mappings $\gamma_\vector{r}^{(i)}$, $i\in\set{1,\ldots,4}$ are contractive apart from a finite subset of $\Zi$. Therefore the above iteration becomes stationary eventually \cite{ABBPTI}. Let $\Pi_\vector{r}$ - the {\em graph of witnesses associated with $\vector{r}$} - denote the edge-colored multidigraph with vertex set $V_\vector{r}$ having an edge of color $i$ from $\vector{a}$ to $\vector{b}$ iff $\gamma_\vector{r}^{(i)}(\vector{a}) = \vector{b}$. If $E_i$ is the set of all edges (ordered pairs) of color $i$ then the graph $\Pi_\vector{r}$ is completely characterized by $(E_1,\ldots,E_4) \in \powerset((\Zi^d)^2)^4$\footnote{$\powerset(M)$ denotes the power set of a set $M$.} (as there are no isolated vertices) and thus the graph and the $4$-tuple can be identified. For any such graph $\Pi = (E_1,\ldots,E_4) \in \powerset((\Zi^d)^2)^4$ let - just as for cycles - $P(\Pi) \ce \set{\vector{r} \in \R^d \mid \fa i\in\set{1,\ldots,4}:\fa (\vector{a},\vector{b}) \in E_i: \gamma_\vector{r}^{(i)}(\vector{a})=\vector{b}}$ and $P_\vector{r} \ce P(\Pi_\vector{r})$. If $\vector{r} \in \Int{\G{d}}$ then $\Pi_\vector{r}$ is finite and $P_\vector{r}$ is a convex polyhedron. Furthermore $\GN{d}$ is the disjoint union of those $P_\vector{r}$ the corresponding parameters $\vector{r}$ of which belong to $\GN{d}$ \cite{Weitzer15}. The algorithms introduced in \cite{Weitzer15} are based on this fact and can easily be adapted to the complex case. The results presented in Section~\ref{SInclusion2} have been achieved in this way.

For further considerations it should be noted that the Loudspeaker is symmetric with respect to the real axis \cite{BrunotteKirschenhoferThuswaldner11}.

\section{One inclusion}
\label{SInclusion1}
\noindent
\begin{tabbing}
In the following we define 19 infinite families of cycles the corresponding cutout\\
polygons of which will cut out the whole region outside $\GC$. Let
\end{tabbing}
\begingroup
\fontsize{6pt}{6pt}\selectfont
\begin{ntabbing}
$\alpha(a,b)$\=\ce\kill
$\iota(a,b)$\>$\ce(n+3k+a,-n-3m+3k+b)$\qquad\=$\alpha(a,b)$\=\kill
$\alpha(a,b)$\>$\ce(n+a,-3m+b)$\>$\beta(a,b)$\>$\ce(n+m+a,b)$\\
$\gamma(a,b)$\>$\ce(n+m-k+a,3k+b)$\>$\delta(a,b)$\>$\ce(n+3m+k+a,-n+3k+b)$\\
$\varepsilon(a,b)$\>$\ce(n+3m-k+a,3k+b)$\>$\zeta(a,b)$\>$\ce(n+3m-3k+a,n+3k+b)$\\
$\eta(a,b)$\>$\ce(n+k+a,-3m+3k+b)$\>$\vartheta(a,b)$\>$\ce(n-k+a,3m+k+b)$\\
$\iota(a,b)$\>$\ce(n+3k+a,-n-3m+3k+b)$\>$\kappa(a,b)$\>$\ce(n-3k+a,n+3m+k+b)$\\
$\lambda(a,b)$\>$\ce(4n/3+3m+a,b)$\>$\mu(a,b)$\>$\ce(4n/3+3m-k+a,3k+b)$\\
$\nu(a,b)$\>$\ce(3m+a,n+b)$\>$\xi(a,b)$\>$\ce(3m+k+a,-n+k+b)$\\
$\varrho(a,b)$\>$\ce(3m-3k+a,n+k+b)$\>$\sigma(a,b)$\>$\ce(3k+a,-n-m+k+b)$\\
$\tau(a,b)$\>$\ce(3k+a,-4n/3-3m+k+b)$\\
\\
$\Gamma(a)\ce\set{1,\ldots,(n+a)/3}\qquad\Sigma(a)\ce\set{1,\ldots,n-3m-a}\qquad\Theta(a)\ce\set{1,\ldots,m+a}$\\
\end{ntabbing}
\begin{ntabbing}
\= $\mathbf{C_0(10)} \ce$ \= $($\=\kill
\>$\mathbf{C_0(1)} \ce$ \> $((-2, 0), (2, 2), (0, -2), (-1, 2), (2, 0), (-1, -1), (0, 2), (2, -1))$\\
\>$\mathbf{C_0(2)} \ce$ \> $((-3, 0), (3, 2), (-1, -2), (1, 3), (1, -3), (-2, 3), (3, -1))$\\
\>$\mathbf{C_0(3)} \ce$ \> $((-4, 0), (4, 2), (-3, -3), (2, 4), (0, -4), (-1, 4), (3, -3), (-3, 2), (4, -1))$\\
\> $\mathbf{C_0(4)} \ce$ \> $((-3,0),(3,3),(0,-4),(-2,3),(4,0),(-2,-2),(0,3),(3,-2))$\\
\>$\mathbf{C_0(5)} \ce$ \> $((-3,0),(3,3),(0,-4),(-2,3),(4,0),(-2,-2),(1,3),(2,-2),(-2,1),(3,1),(-1,-2),$\\
\>\>\>$(0,3),(3,-2))$\\
\>$\mathbf{C_0(6)} \ce$ \> $((-3,-1),(3,3),(-1,-3),(0,4),(2,-3),(-3,2),(4,0))$\\
\>$\mathbf{C_0(7}) \ce$ \> $((-4,-2),(3,4),(0,-4),(-1,4),(3,-3),(-4,2),(5,0))$\\
\>$\mathbf{C_0(8)} \ce$ \> $((-5,-1),(5,3),(-3,-4),(2,5),(0,-5),(-1,5),(3,-4),(-4,3),(5,-1),(-5,0),(5,2),$\\
\>\>\>$(-4,-3),(3,5),(-1,-5),(0,6),(2,-5),(-3,5),(5,-3),(-5,2),(6,0))$\\
\>$\mathbf{C_0(9)} \ce$ \> $((-5,0),(5,2),(-4,-3),(3,5),(-1,-5),(0,5),(2,-4),(-3,4),(5,-2),(-5,1),(5,1),$\\
\>\>\>$(-4,-2),(4,4),(-2,-4),(1,5),(1,-5),(-2,5),(4,-4),(-4,3),(5,-1))$\\
\>$\mathbf{C_0(10)} \ce$ \> $((-15,-5),(13,10),(-9,-13),(5,15),(0,-15),(-4,15),(9,-12),(-12,9),(15,-4),(-15,0),$\\
\>\>\>$(15,5),(-12,-9),(9,13),(-4,-15),(0,16),(5,-15),(-9,13),(13,-9),(-15,5),(16,0))$\\
\>$\mathbf{C_0(11)} \ce$ \> $((-4,0),(4,2),(-3,-2),(3,3),(-2,-3),(2,4),(0,-4),(-1,4),(3,-3),(-3,2),(4,-1))$\\
\>$\mathbf{C_0(12)} \ce$ \> $((-7,0),(7,2),(-6,-3),(6,5),(-4,-5),(3,6),(-1,-6),(0,7),(2,-6),(-3,6),(5,-5),$\\
\>\>\>$(-5,4),(6,-2),(-6,1),(7,1),(-6,-2),(6,4),(-5,-5),(4,6),(-2,-6),(1,7),(1,-7),$\\
\>\>\>$(-2,7),(4,-6),(-5,6),(6,-4),(-6,3),(7,-1))$\\
\>$\mathbf{C_0(13)} \ce$ \> $((-7,-1),(7,3),(-6,-4),(6,6),(-4,-6),(3,7),(-1,-7),(0,8),(2,-7),(-3,7),(5,-6),$\\
\>\>\>$(-6,5),(7,-3),(-7,2),(8,0))$\\
\>$\mathbf{C_0(14)} \ce$ \> $((-10,0),(10,2),(-9,-3),(9,5),(-7,-6),(6,8),(-4,-8),(3,9),(-1,-9),(0,10),(2,-9),$\\
\>\>\>$(-3,9),(5,-8),(-6,7),(8,-5),(-8,4),(9,-2),(-9,1),(10,1),(-9,-2),(9,4),(-8,-5),$\\
\>\>\>$(7,7),(-5,-8),(4,9),(-2,-9),(1,10),(1,-10),(-2,10),(4,-9),(-5,9),(7,-7),(-8,6),$\\
\>\>\>$(9,-4),(-9,3),(10,-1))$
\end{ntabbing}
\begin{ntabbing}
$\quad$ \= \kill
$\mathbf{C_1(n, m)}\ce
(-\beta(0,0))
(\gamma(1,-1), -\gamma(0,0))_{k\in\Theta(1)}
(\vartheta(0,3), -\vartheta(1,-3))_{k\in\Sigma(5)}
(\varrho(7,-2), -\varrho(-5,2))_{k\in\Theta(2)}$\\
$(-\sigma(1,1), \sigma(1,0))_{k\in\Theta(1)}
(-\xi(-3,0), \xi(4,1))_{k\in\Sigma(5)}
(-\eta(2,6), \eta(-1,-4))_{k\in\Theta(1)}$\\
for $n \geq 2 \quad\land\quad -1 \leq m \leq (n - 5) / 3$\\
$\mathbf{C_2(n, m)}\ce
(-\gamma(-1,1), \gamma(1,1))_{k\in\Theta(1)}
(-\vartheta(0,-4), \vartheta(0,5))_{k\in\Sigma(6)}
(-\varrho(-7,2), \varrho(6,-1))_{k\in\Theta(2)}$\\
$(\sigma(-1,-1), -\sigma(0,1))_{k\in\Theta(1)}
(\xi(4,0), -\xi(-4,0))_{k\in\Sigma(6)}
(\eta(-2,-8), -\eta(2,7))_{k\in\Theta(2)}
(\beta(1,1))$\\
for $n \geq 3 \quad\land\quad -1 \leq m \leq (n - 6) / 3$\\
$\mathbf{C_3(n, m)}\ce
(-\gamma(-1,3), \gamma(1,-1))_{k\in\Theta(1)}
(-\vartheta(0,-2), \vartheta(0,3))_{k\in\Sigma(4)}
(-\nu(-2,1), \nu(2,0))(\text{if }m = 0)$\\
$(-\nu(-2,1), \nu(2,0), -\nu(0,0), \nu(-1,1))(\text{if }m \neq 0)
(-\varrho(0,0), \varrho(-1,1))_{k\in\Theta(-1)}
(\sigma(-3,-1), -\sigma(2,1))_{k\in\Theta(1)}$\\
$(\xi(2,0), -\xi(-2,0))_{k\in\Sigma(4)}
(\alpha(-1,-3), -\alpha(1,2))(\text{if }m = 0)
(\alpha(-1,-3), -\alpha(1,2), \alpha(0,-1), -\alpha(0,0))(\text{if }m \neq 0)$\\
$(\eta(0,-1), -\eta(0,0))_{k\in\Theta(-1)}
(\beta(0,-1))$\\
for $n \geq 5 \quad\land\quad 0 \leq m \leq (n - 5) / 3$\\
$\mathbf{C_4(n, m)}\ce
(-\gamma(-1,2), \gamma(1,0))_{k\in\Theta(1)}
(-\vartheta(0,-3), \vartheta(0,4))_{k\in\Sigma(5)}
(-\nu(-3,1), \nu(3,0))$\\
$(-\varrho(-4,1), \varrho(3,0))_{k\in\Theta(1)}
(\sigma(-1,-1), -\sigma(0,1))_{k\in\Theta(1)}
(\xi(4,0), -\xi(-4,0))_{k\in\Sigma(6)}
(\alpha(-1,-5), -\alpha(1,4))$\\
$(\eta(-1,-6), -\eta(1,5))_{k\in\Theta(1)}
(\beta(1,0))$\\
for $n \geq 6 \quad\land\quad 0 \leq m \leq (n - 6) / 3$\\
$\mathbf{C_5(n, m)}\ce
(-\gamma(-1,3), \gamma(1,-1))_{k\in\Theta(1)}
(-\vartheta(0,-1), \vartheta(0,2))_{k\in\Sigma(2)}
(-\varrho(-3,1), \varrho(2,0))_{k\in\Theta(0)}$\\
$(\sigma(-3,-1), -\sigma(2,1))_{k\in\Theta(1)}
(\xi(1,0), -\xi(-1,0))_{k\in\Sigma(2)}
(\eta(-1,-4), -\eta(1,3))_{k\in\Theta(0)}
(\beta(0,-1))$\\
for $n \geq 2 \quad\land\quad 0 \leq m \leq (n - 2) / 3$\\
$\mathbf{C_6(n, m)}\ce
(-\gamma(-1,2), \gamma(1,0))_{k\in\Theta(1)}
(-\vartheta(0,-2), \vartheta(0,3))_{k\in\Sigma(3)}
(-\varrho(-4,1), \varrho(3,0))_{k\in\Theta(1)}$\\
$(\sigma(-1,-1), -\sigma(0,1))_{k\in\Theta(1)}
(\xi(3,0), -\xi(-3,0))_{k\in\Sigma(4)}
(\eta(-1,-6), -\eta(1,5))_{k\in\Theta(1)}
(\beta(1,0))$\\
for $n \geq 4 \quad\land\quad 0 \leq m \leq (n - 4) / 3$\\
$\mathbf{C_7(n, m)}\ce
(-\beta(0,1))
(\gamma(1,-2), -\gamma(0,1))_{k\in\Theta(1)}
(\vartheta(0,2), -\vartheta(1,-2))_{k\in\Sigma(3)}
(\varrho(5,-1), -\varrho(-3,1))_{k\in\Theta(1)}$\\
$(-\sigma(2,1), \sigma(0,0))_{k\in\Theta(1)}
(-\xi(-2,0), \xi(3,1))_{k\in\Sigma(3)}
(-\eta(1,4), \eta(0,-2))_{k\in\Theta(0)}$\\
for $n \geq 5 \quad\land\quad 0 \leq m \leq (n - 5) / 5$\\
$\mathbf{C_8(n, m)}\ce
(-\gamma(-1,2), \gamma(1,0))_{k\in\Theta(1)}
(-\vartheta(0,-3), \vartheta(0,4))_{k\in\Sigma(4)}
(-\varrho(-5,1), \varrho(4,0))_{k\in\Theta(1)}$\\
$(\sigma(-2,-1), -\sigma(1,1))_{k\in\Theta(1)}
(\xi(3,0), -\xi(-3,0))_{k\in\Sigma(4)}
(\eta(-1,-6), -\eta(1,5))_{k\in\Theta(1)}
(\beta(1,0))$\\
for $n \geq 1 \quad\land\quad (m = -1 \quad\lor\quad 0 \leq m \leq (n - 8) / 5 \quad\lor\quad m = (n - 4) / 3)$\\
$\mathbf{C_9(n, m)}\ce
(-\beta(0,0))
(\gamma(1,-1), -\gamma(0,0))_{k\in\Theta(0)}
(\vartheta(1,1), -\vartheta(0,-1))_{k\in\Sigma(3)}
(\nu(3,-1), -\nu(-1,1))$\\
$(\varrho(4,-1), -\varrho(-2,1))_{k\in\Theta(1)}
(-\sigma(1,1), \sigma(1,0))_{k\in\Theta(0)}
(-\xi(-1,1), \xi(2,0))_{k\in\Sigma(3)}
(-\alpha(1,2), \alpha(0,-1))$\\
$(-\eta(1,3), \eta(0,-1))_{k\in\Theta(0)}$\\
for $n \geq 4 \quad\land\quad (n - 4) / 5 \leq m \leq (n - 4) / 3$\\
$\mathbf{C_{10}(n, m)}\ce
(-\gamma(-1,1), \gamma(1,1))_{k\in\Theta(1)}
(-\vartheta(0,-3), \vartheta(0,4))_{k\in\Sigma(5)}
(-\nu(-3,1), \nu(3,0))$\\
$(-\varrho(-4,1), \varrho(3,0))_{k\in\Theta(1)}
(\sigma(-1,-1), -\sigma(0,1))_{k\in\Theta(1)}
(\xi(3,0), -\xi(-3,0))_{k\in\Sigma(5)}
(\alpha(-1,-4), -\alpha(1,3))$\\
$(\eta(-1,-5), -\eta(1,4))_{k\in\Theta(1)}
(\beta(1,1))$\\
for $n \geq 5 \quad\land\quad (n - 6) / 5 \leq m \leq (n - 5) / 3$\\
$\mathbf{C_{11}(n, m)}\ce
(-\beta(0,-2))
(\gamma(1,1), -\gamma(0,-2))_{k\in\Theta(0)}
(\vartheta(1,2), -\vartheta(0,-2))_{k\in\Sigma(4)}
(\nu(4,-1), -\nu(-2,1))$\\
$(\varrho(5,-1), -\varrho(-3,1))_{k\in\Theta(1)}
(-\sigma(2,1), \sigma(0,0))_{k\in\Theta(1)}
(-\xi(-2,0), \xi(3,1))_{k\in\Sigma(3)}
(-\eta(1,4), \eta(0,-2))_{k\in\Theta(1)}$\\
for $n \geq 4 \quad\land\quad (n - 5) / 5 \leq m \leq (n - 4) / 3$\\
$\mathbf{C_{12}(n, m)}\ce
(-\gamma(-1,1), \gamma(1,1))_{k\in\Theta(1)}
(-\vartheta(0,-3), \vartheta(0,4))_{k\in\Sigma(5)}
(-\nu(-3,1), \nu(3,0))$\\
$(-\varrho(-4,1), \varrho(3,0))_{k\in\Theta(1)}
(\sigma(-1,-1), -\sigma(0,1))_{k\in\Theta(1)}
(\xi(4,0), -\xi(-4,0))_{k\in\Sigma(6)}$\\
$(\eta(-2,-8), -\eta(2,7))_{k\in\Theta(2)}
(\beta(1,1))$\\
for $n \geq 6 \quad\land\quad (n - 7) / 5 \leq m \leq (n - 6) / 3$\\
$\mathbf{C_{13}(n, m)}\ce
(-\gamma(-1,3), \gamma(1,-1))_{k\in\Theta(1)}
(-\vartheta(0,-1), \vartheta(0,2))_{k\in\Sigma(3)}
(-\nu(-1,1), \nu(1,0))$\\
$(-\varrho(-2,1), \varrho(1,0))_{k\in\Theta(0)}
(\sigma(-2,-1), -\sigma(1,1))_{k\in\Theta(0)}
(\xi(0,-1), -\xi(0,1))_{k\in\Sigma(1)}$\\
$(\eta(-1,-4), -\eta(1,3))_{k\in\Theta(0)}
(\beta(0,-1))$\\
for $n \geq 3 \quad\land\quad (n - 3) / 5 \leq m \leq (n - 3) / 3$\\
$\mathbf{C_{14}(n, m)}\ce
(-\gamma(-1,2), \gamma(1,0))_{k\in\Theta(1)}
(-\vartheta(0,-3), \vartheta(0,4))_{k\in\Sigma(4)}
(-\varrho(-5,1), \varrho(4,0))_{k\in\Theta(1)}$\\
$(\sigma(-2,-2), -\sigma(1,2))_{k\in\Theta(1)}
(\xi(3,-1), -\xi(-3,1))_{k\in\Sigma(5)}
(\alpha(-1,-5), -\alpha(1,4))$\\
$(\eta(-1,-6), -\eta(1,5))_{k\in\Theta(1)}
(\beta(1,0))$\\
for $n \geq 2 \quad\land\quad (n - 7) / 5 \leq m \leq (n - 5) / 3$\\
$\mathbf{C_{15}(n, m)}\ce
(-\gamma(-1,2), \gamma(1,0))_{k\in\Theta(1)}
(-\vartheta(0,-3), \vartheta(0,4))_{k\in\Sigma(5)}
(-\nu(-3,1), \nu(3,0))$\\
$(-\varrho(-4,1), \varrho(3,0))_{k\in\Theta(1)}
(\sigma(-1,-1), -\sigma(0,1))_{k\in\Theta(1)}
(\xi(3,0), -\xi(-3,0))_{k\in\Sigma(4)}$\\
$(\eta(-1,-6), -\eta(1,5))_{k\in\Theta(1)}
(\beta(1,0))$\\
for $n \geq 4 \quad\land\quad ((n - 7) / 5 \leq m \leq (n - 6) / 3 \quad\lor\quad m = (n - 4) / 3)$\\
$\mathbf{C_{16}(n, m)}\ce
(-\gamma(-1,3), \gamma(1,-1))_{k\in\Theta(1)}
(-\vartheta(0,-1), \vartheta(0,2))_{k\in\Sigma(2)}
(-\varrho(-3,1), \varrho(2,0))_{k\in\Theta(0)}$\\
$(\sigma(-3,-1), -\sigma(2,1))_{k\in\Theta(1)}
(\xi(2,0), -\xi(-2,0))_{k\in\Sigma(4)}
(\alpha(-1,-3), -\alpha(1,2), \alpha(0,-1), -\alpha(0,0))$\\
$(\eta(0,-1), -\eta(0,0))_{k\in\Theta(-1)}
(\beta(0,-1))$\\
for $n \geq 7 \quad\land\quad (n - 4) / 5 \leq m \leq (n - 4) / 3$\\
$\mathbf{C_{17}(n, m)}\ce
(-\gamma(-1,1), \gamma(1,1))_{k\in\Theta(0)}
(-\vartheta(-1,-1), \vartheta(1,2))_{k\in\Sigma(2)}
(-\varrho(-4,1), \varrho(3,0))_{k\in\Theta(1)}$\\
$(\sigma(-1,-1), -\sigma(0,1))_{k\in\Theta(1)}
(\xi(3,0), -\xi(-3,0))_{k\in\Sigma(5)}
(\alpha(-1,-4), -\alpha(1,3))$\\
$(\eta(-1,-5), -\eta(1,4))_{k\in\Theta(1)}
(\beta(1,1))$\\
for $n \geq 5 \quad\land\quad (n - 5) / 5 \leq m \leq (n - 5) / 3$\\
$\mathbf{C_{18}(n, m)}\ce
(-\gamma(-1,3), \gamma(1,-1))_{k\in\Theta(1)}
(-\vartheta(0,-2), \vartheta(0,3))_{k\in\Sigma(4)}
(-\varrho(-5,2), \varrho(4,-1))_{k\in\Theta(1)}$\\
$(\sigma(-2,-1), -\sigma(1,1))_{k\in\Theta(1)}
(\xi(2,0), -\xi(-2,0))_{k\in\Sigma(4)}
(\alpha(-1,-3), -\alpha(1,2))$\\
$(\eta(-1,-4), -\eta(1,3))_{k\in\Theta(0)}
(\beta(0,-1))$\\
for $n \geq 8 \quad\land\quad (n - 4) / 5 \leq m \leq (n - 5) / 3$\\
$\mathbf{C_{19}(n, m)}\ce
(-\mu(0,1), \mu(0,1))_{k\in\Gamma(-3)}
(-\zeta(-3,4), \zeta(2,-2))_{k\in\Theta(0)}
(-\kappa(-3,2), \kappa(2,-1))_{k\in\Gamma(0)}$\\
$(\tau(-3,0), -\tau(2,0))_{k\in\Gamma(0)}
(\iota(-3,-2), -\iota(2,1))_{k\in\Theta(0)}
(\delta(-1,-2), -\delta(1,1))_{k\in\Gamma(0)}
(\lambda(0,1))$\\
for $n \geq 3 \quad\land\quad n \equiv 0 \pmod{3} \quad\land\quad (m = 0 \quad\lor\quad 1 \leq m \leq (n - 3 / 2) \cdot 2 / 9)$\\
$(-\varepsilon(0,1), \mu(2/3,1))_{k\in\Gamma(-1)}
(-\zeta(-2,2), \zeta(1,0))_{k\in\Theta(0)}
(-\kappa(-2,0), \kappa(1,1))_{k\in\Gamma(-1)}$\\
$
(\tau(-3,-5/3), -\tau(-2,-5/3))_{k\in\Gamma(-1)}
(\iota(-4,-4), -\iota(3,3))_{k\in\Theta(1)}
(\delta(0,-1), -\delta(0,0))_{k\in\Gamma(-1)}
(\lambda(2/3,1))$\\
for $n \geq 1 \quad\land\quad n \equiv 1 \pmod{3} \quad\land\quad (m = 0 \quad\lor\quad 1 \leq m \leq (n - 5 / 2 ) \cdot 2 / 9)$\\
$(-\mu(7/3,-3), \mu(7/3,-1))_{k\in\Gamma(1)}
(-\zeta(-3,2), \zeta(2,0))_{k\in\Theta(0)}
(-\kappa(-3,0), \kappa(2,1))_{k\in\Gamma(1)}$\\
$
(\tau(-2,-7/3), -\tau(-1,-7/3))_{k\in\Gamma(-2)}
(\iota(-4,-5), -\iota(3,4))_{k\in\Theta(1)}
(\delta(1,-2), -\delta(-1,1))_{k\in\Gamma(-2)}
(\lambda(4/3,-1))$\\
for $n \geq 5 \quad\land\quad n \equiv 2 \pmod{3} \quad\land\quad (m = 0 \quad\lor\quad 1 \leq m \leq (n - 7 / 2) \cdot 2 / 9)$
\end{ntabbing}
\endgroup
\phantom{a}
\noindent
By Lemma~5.1 of \cite{Weitzer15} one can compute the corresponding cutout polygons.

\phantom{a}
\begingroup
\fontsize{6pt}{9pt}\selectfont
\begin{ntabbing}
$\mathbf{C_0(n)}:$\\
$n=11: \quad$ \= \kill
$n=1:$ \> $\overline{(\frac{2}{3}, \frac{2}{3})}$\\
$n=2:$ \> $(\frac{12}{13}, \frac{5}{13})\MDots(\frac{6}{7}, \frac{3}{7})\MLine\overline{(\frac{7}{8}, \frac{3}{8})}\MLine(\frac{10}{11}, \frac{4}{11})\MDots$\\
$n=3:$ \> $(1, \frac{1}{3})\MLine\overline{(\frac{13}{14}, \frac{2}{7})}\MLine\overline{(\frac{17}{18}, \frac{5}{18})}\MLine$\\
$n=4:$ \> $(\frac{3}{4},\frac{3}{4})\MLine(\frac{2}{3},\frac{2}{3})\MLine$\\
$n=5:$ \> $(\frac{3}{4},\frac{2}{3})\MDots(\frac{3}{4},\frac{3}{4})\MDots(\frac{2}{3},\frac{2}{3})\MDots$\\
$n=6:$ \> $(\frac{14}{15},\frac{2}{5})\MDots(\frac{5}{6},\frac{1}{2})\MDots(\frac{6}{7},\frac{3}{7})\MLine\overline{(\frac{12}{13},\frac{5}{13})}\MLine$\\
$n=7:$ \> $(\frac{23}{25},\frac{11}{25})\MDots(\frac{10}{11},\frac{5}{11})\MLine\overline{(\frac{8}{9},\frac{4}{9})}\MLine(\frac{19}{21},\frac{3}{7})\MDots$\\
$n=8:$ \> $(1,\frac{1}{3})\MDots(\frac{14}{15},\frac{1}{3})\MDots(\frac{16}{17},\frac{5}{17})\MDots(\frac{24}{25},\frac{7}{25})\MLine$\\
$n=9:$ \> $(\frac{16}{17},\frac{5}{17})\MLine(\frac{15}{16},\frac{5}{16})\MLine$\\
$n=10:$ \> $\overline{(\frac{15}{16},\frac{5}{16})}$\\
$n=11:$ \> $(\frac{17}{18},\frac{5}{18})\MLine(\frac{14}{15},\frac{4}{15})\MLine$\\
$n=12:$ \> $(\frac{48}{49},\frac{9}{49})\MDots(\frac{36}{37},\frac{7}{37})\MLine\overline{(\frac{37}{38},\frac{7}{38})}\MLine(\frac{43}{44},\frac{2}{11})\MDots$\\
$n=13:$ \> $(\frac{65}{66},\frac{2}{11})\MDots(\frac{35}{36},\frac{7}{36})\MDots(\frac{36}{37},\frac{7}{37})\MLine\overline{(\frac{60}{61},\frac{11}{61})}\MLine$\\
$n=14:$ \> $(\frac{87}{88},\frac{2}{11})\MDots(\frac{51}{52},\frac{5}{26})\MLine\overline{(\frac{57}{58},\frac{5}{29})}\MLine$
\end{ntabbing}
\begin{ntabbing}
$\quad$ \= \kill
$\mathbf{C_1(n,m)}:$\\
$n=2\land m=-1:(1,1)\MLine(0,1)\MLine(\frac{1}{2},\frac{1}{2})\MDots$\\
$n=3\land m=-1:(1,\frac{1}{2})\MLine(\frac{3}{4},\frac{1}{2})\MLine(\frac{4}{5},\frac{2}{5})\MDots$\\
$n\geq 4\land m=-1:(1,\frac{1}{n-1})\MLine\overline{(1-\frac{1}{n^2-2 n+1},\frac{n-1}{n^2-2 n+1})}\MLine(1-\frac{1}{n^2-2 n+2},\frac{n-1}{n^2-2 n+2})\MDots$\\
$n=5\land m=0:(1,\frac{1}{4})\MLine(\frac{24}{25},\frac{7}{25})\MDots(\frac{18}{19},\frac{5}{19})\MLine\overline{(\frac{21}{22},\frac{5}{22})}\MLine$\\
$n\geq 9\land 0\leq m\leq \frac{n-9}{5}:(1,\frac{1}{n-1})\MLine\overline{(1-\frac{1}{n^2-n+n m-4 m-3},\frac{n+m}{n^2-n+n m-4 m-3})}\MLine$ \\\> $\overline{(1-\frac{1}{n^2-n+n m+2 m+2},\frac{n+m}{n^2-n+n m+2 m+2})}\MLine$\\
$n\geq 6\land \frac{n-8}{5}\leq m\leq \frac{n-5}{5}:(1,\frac{1}{n-1})\MLine(1-\frac{1}{8 n+6 n m-9 m-11},\frac{6 m+8}{8 n+6 n m-9 m-11})\MDots$ \\\> $(1-\frac{1}{n^2-2 n+n m+m+5},\frac{n+m}{n^2-2 n+n m+m+5})\MLine\overline{(1-\frac{1}{n^2-n+n m+2 m+2},\frac{n+m}{n^2-n+n m+2 m+2})}\MLine$\\
$n\geq 9\land \frac{n-4}{5}\leq m\leq \frac{n-6}{3}:(1,\frac{1}{n-1})\MLine(1-\frac{1}{8 n+6 n m-9 m-11},\frac{6 m+8}{8 n+6 n m-9 m-11})\MDots$ \\\> $(1-\frac{1}{n^2-2 n+n m+m+5},\frac{n+m}{n^2-2 n+n m+m+5})\MLine(1-\frac{1}{n^2+n m-3 m-2},\frac{n+m}{n^2+n m-3 m-2})\MDots$ \\\> $(1-\frac{1}{4 n+6 n m-3 m-2},\frac{6 m+4}{4 n+6 n m-3 m-2})\MLine$\\
$n\geq 8\land m=\frac{n-5}{3}:(1,\frac{1}{n-1})\MLine(1-\frac{1}{2 n^2-6 n+5},\frac{2 n-3}{2 n^2-6 n+5})\MDots(1-\frac{3}{4 n^2-10 n+7},\frac{4 n-5}{4 n^2-10 n+7})\MLine$\\\>$(1-\frac{3}{4 n^2-8 n+9},\frac{4 n-5}{4 n^2-8 n+9})\MDots(1-\frac{1}{2 n^2-7 n+3},\frac{2 n-6}{2 n^2-7 n+3})\MLine$
\\
$\mathbf{C_2(n,m)}:$\\
$n=3\land m=-1:(\frac{7}{8},\frac{5}{8})\MDots(\frac{5}{6},\frac{2}{3})\MLine(\frac{2}{3},\frac{2}{3})\MDots(\frac{4}{5},\frac{2}{5})\MDots$\\
$n=4\land m=-1:(1,\frac{1}{3})\MDots(\frac{12}{13},\frac{5}{13})\MDots(\frac{8}{9},\frac{1}{3})\MDots(\frac{9}{10},\frac{3}{10})\MDots$\\
$n\geq 7\land -1\leq m\leq \frac{n-12}{5}:(1,\frac{1}{n-1})\MDots(1-\frac{1}{n^2-n+n m-4 m-5},\frac{n+m}{n^2-n+n m-4 m-5})\MDots$ \\\> $(1-\frac{1}{n^2-n+n m+2 m+4},\frac{n+m}{n^2-n+n m+2 m+4})\MDots$\\
$n\geq 5\land \frac{n-11}{5}\leq m\leq \frac{n-8}{5}:(1,\frac{1}{n-1})\MDots(1-\frac{1}{12 n+6 n m-9 m-17},\frac{6 m+12}{12 n+6 n m-9 m-17})\MLine$ \\\> $(1-\frac{1}{n^2-2 n+n m+m+7},\frac{n+m}{n^2-2 n+n m+m+7})\MDots(1-\frac{1}{n^2-n+n m+2 m+4},\frac{n+m}{n^2-n+n m+2 m+4})\MDots$\\
$n\geq 11\land \frac{n-7}{5}\leq m\leq \frac{n-8}{3}:(1,\frac{1}{n-1})\MDots(1-\frac{1}{12 n+6 n m-9 m-17},\frac{6 m+12}{12 n+6 n m-9 m-17})\MLine$ \\\> $(1-\frac{1}{n^2-2 n+n m+m+7},\frac{n+m}{n^2-2 n+n m+m+7})\MDots(1-\frac{1}{n^2+n m-3 m-4},\frac{n+m}{n^2+n m-3 m-4})\MLine$ \\\> $(1-\frac{1}{8 n+6 n m-3 m-4},\frac{6 m+8}{8 n+6 n m-3 m-4})\MDots$\\
$n\geq 7\land m=\frac{n-7}{3}:(1,\frac{1}{n-1})\MDots(1-\frac{1}{2 n^2-6 n+5},\frac{2 n-3}{2 n^2-6 n+5})\MDots(1-\frac{3}{4 n^2-12 n+11},\frac{4 n-7}{4 n^2-12 n+11})\MDots$\\
\> $(1-\frac{3}{4 n^2-10 n+9},\frac{4 n-7}{4 n^2-10 n+9})\MLine(1-\frac{1}{2 n^2-7 n+3},\frac{2 n-6}{2 n^2-7 n+3})\MDots$\\
$n\geq 6\land m=\frac{n-6}{3}:(1-\frac{1}{2 n^2-4 n+2},\frac{2 n-1}{2 n^2-4 n+2})\MDots(1-\frac{3}{4 n^2-13 n+9},\frac{4 n-6}{4 n^2-13 n+9})\MDots$ \\\> $(1-\frac{3}{4 n^2-11 n+12},\frac{4 n-6}{4 n^2-11 n+12})\MDots$
\\
$\mathbf{C_3(n,m)}:$\\
$n\geq 5\land 0\leq m\leq \frac{n-5}{5}:(1,\frac{1}{n-1})\MDots(1-\frac{1}{n^2-n+n m+2 m+2},\frac{n+m}{n^2-n+n m+2 m+2})\MLine$ \\\> $\overline{(1-\frac{1}{n^2-n+n m+2 m+3},\frac{n+m}{n^2-n+n m+2 m+3})}\MLine$\\
$n\geq 8\land \frac{n-4}{5}\leq m\leq \frac{n-5}{3}:(1,\frac{1}{n-1})\MDots(1-\frac{1}{4 n+6 n m-3 m-2},\frac{6 m+4}{4 n+6 n m-3 m-2})\MLine$ \\\> $\overline{(1-\frac{1}{5 n+6 n m-3 m-2},\frac{6 m+5}{5 n+6 n m-3 m-2})}\MLine$
\\
$\mathbf{C_4(n,m)}:$\\
$n\geq 8\land 0\leq m\leq \frac{n-8}{5}:(1,\frac{1}{n-1})\MLine(1-\frac{1}{n^2-n+n m+2 m+4},\frac{n+m}{n^2-n+n m+2 m+4})\MDots$ \\\> $(1-\frac{1}{n^2-n+n m+2 m+5},\frac{n+m}{n^2-n+n m+2 m+5})\MDots$\\
$n\geq 6\land \frac{n-7}{5}\leq m\leq \frac{n-6}{3}:(1,\frac{1}{n-1})\MLine(1-\frac{1}{7 n+6 n m-3 m-3},\frac{6 m+7}{7 n+6 n m-3 m-3})\MDots$ \\\> $(1-\frac{1}{8 n+6 n m-3 m-3},\frac{6 m+8}{8 n+6 n m-3 m-3})\MDots$
\\
$\mathbf{C_5(n,m)}:$\\
$2\leq n\leq 3\land m=0:(1,\frac{1}{n})\MLine\overline{(1-\frac{1}{n^2+n-1},\frac{n+1}{n^2+n-1})}\MLine(1-\frac{1}{n^2},\frac{n}{n^2})\MDots$\\
$n\geq 4\land m=0:(1,\frac{1}{n})\MLine\overline{(1-\frac{1}{n^2-1},\frac{n}{n^2-1})}\MLine(1-\frac{1}{n^2},\frac{n}{n^2})\MDots$\\
$n\geq 9\land 1\leq m\leq \frac{n-4}{5}:(1,\frac{1}{n})\MLine\overline{(1-\frac{1}{n^2+n m-3 m-1},\frac{n+m}{n^2+n m-3 m-1})}\MLine$ \\\> $(1-\frac{1}{n^2+n m-3 m},\frac{n+m}{n^2+n m-3 m})\MDots$\\
$n\geq 6\land \frac{n-3}{5}\leq m\leq \frac{n-3}{3}:(1,\frac{1}{n})\MLine\overline{(1-\frac{1}{4 n+6 n m-3 m-1},\frac{6 m+4}{4 n+6 n m-3 m-1})}\MLine$ \\\> $(1-\frac{1}{3 n+6 n m-3 m},\frac{6 m+3}{3 n+6 n m-3 m})\MDots$\\
$n\geq 5\land m=\frac{n-2}{3}:(1,\frac{1}{n})\MLine\overline{(1-\frac{1}{2 n^2-2 n+1},\frac{2 n-1}{2 n^2-2 n+1})}\MLine(1-\frac{1}{2 n^2-3 n+2},\frac{2 n-2}{2 n^2-3 n+2})\MDots$
\\
$\mathbf{C_6(n,m)}:$\\
$n\geq 7\land 0\leq m\leq \frac{n-7}{5}:(1,\frac{1}{n})\MDots(1-\frac{1}{n^2+n m-3 m-3},\frac{n+m}{n^2+n m-3 m-3})\MDots$ \\\> $(1-\frac{1}{n^2+n m-3 m-2},\frac{n+m}{n^2+n m-3 m-2})\MLine$\\
$n\geq 5\land \frac{n-6}{5}\leq m\leq \frac{n-5}{3}:(1,\frac{1}{n})\MDots(1-\frac{1}{7 n+6 n m-3 m-3},\frac{6 m+7}{7 n+6 n m-3 m-3})\MDots$ \\\> $(1-\frac{1}{6 n+6 n m-3 m-2},\frac{6 m+6}{6 n+6 n m-3 m-2})\MLine$\\
$n\geq 4\land m=\frac{n-4}{3}:(1,\frac{1}{n})\MDots(1-\frac{1}{2 n^2-2 n+1},\frac{2 n-1}{2 n^2-2 n+1})\MLine\overline{(1-\frac{1}{2 n^2-3 n+2},\frac{2 n-2}{2 n^2-3 n+2})}\MLine$
\\
$\mathbf{C_7(n,m)}:$\\
$n\geq 9\land 0\leq m\leq \frac{n-9}{11}:(1-\frac{1}{n^2+n m-3 m-1},\frac{n+m}{n^2+n m-3 m-1})\MLine$ \\\> $(1-\frac{1}{n^2-n+n m+2 m+3},\frac{n+m}{n^2-n+n m+2 m+3})\MDots(1-\frac{2}{n^2-n+n m+5 m+6},\frac{n+m}{n^2-n+n m+5 m+6})\MLine$ \\\> $(1-\frac{2}{n^2+n m-6 m-2},\frac{n+m}{n^2+n m-6 m-2})\MDots$\\
$n\geq 5\land \frac{n-8}{11}\leq m\leq \frac{n-5}{5}:(1-\frac{1}{n^2+n m-3 m-1},\frac{n+m}{n^2+n m-3 m-1})\MLine$ \\\> $(1-\frac{1}{n^2-n+n m+2 m+3},\frac{n+m}{n^2-n+n m+2 m+3})\MDots(1-\frac{1}{4 n+6 n m-3 m-1},\frac{6 m+4}{4 n+6 n m-3 m-1})\MDots$
\\
$\mathbf{C_8(n,m)}:$\\
$n=1\land m=-1:(0, 0)\MDots(0, 1)\MDots(-1, 1)\MLine\overline{(-1, 0)}\MLine$\\
$n=2\land m=-1:(\frac{3}{4},\frac{1}{2})\MDots(\frac{2}{3},\frac{2}{3})\MDots(\frac{1}{2},\frac{1}{2})\MLine$\\
$n=3\land m=-1:(\frac{8}{9},\frac{1}{3})\MDots(\frac{7}{8},\frac{3}{8})\MLine(\frac{5}{6},\frac{1}{3})\MLine$\\
$n\geq 4\land m=-1:(1-\frac{1}{n^2},\frac{1}{n})\MDots(1-\frac{1}{n^2-n+2},\frac{n}{n^2-n+2})\MLine(1-\frac{1}{n^2-2 n+3},\frac{n-1}{n^2-2 n+3})\MDots$ \\\> $(1-\frac{1}{n^2-n},\frac{n-1}{n^2-n})\MLine$\\
$n\geq 16\land 0\leq m\leq \frac{n-16}{11}:(1-\frac{1}{n^2+n+n m-3 m-3},\frac{n+m+1}{n^2+n+n m-3 m-3})\MDots$ \\\> $(1-\frac{1}{n^2+n m+2 m+4},\frac{n+m+1}{n^2+n m+2 m+4})\MLine(1-\frac{2}{n^2-n+n m+5 m+10},\frac{n+m}{n^2-n+n m+5 m+10})\MDots$ \\\> $(1-\frac{2}{n^2+n m-6 m-6},\frac{n+m}{n^2+n m-6 m-6})\MLine$\\
$n\geq 8\land \frac{n-15}{11}\leq m\leq \frac{n-8}{5}:(1-\frac{1}{n^2+n+n m-3 m-3},\frac{n+m+1}{n^2+n+n m-3 m-3})\MDots$ \\\> $(1-\frac{1}{n^2+n m+2 m+4},\frac{n+m+1}{n^2+n m+2 m+4})\MLine\overline{(1-\frac{1}{8 n+6 n m-3 m-3},\frac{6 m+8}{8 n+6 n m-3 m-3})}\MLine$\\
$n\geq 4\land m=\frac{n-4}{3}:(1-\frac{3}{4 n^2-4 n+3},\frac{4 n-1}{4 n^2-4 n+3})\MDots(1-\frac{3}{4 n^2-6 n+2},\frac{4 n-1}{4 n^2-6 n+2})\MLine$ \\\> $\overline{(1-\frac{2}{2 n^2-3 n+2},\frac{2 n-1}{2 n^2-3 n+2})}\MLine$
\\
$\mathbf{C_9(n,m)}:$\\
$n\geq 4\land m=\frac{n-4}{5}:(1-\frac{5}{6 n^2-7 n+7},\frac{6 n-4}{6 n^2-7 n+7})\MDots(1-\frac{5}{6 n^2-2 n+2},\frac{6 n+1}{6 n^2-2 n+2})\MDots$ \\\> $(1-\frac{5}{6 n^2-7 n+2},\frac{6 n-4}{6 n^2-7 n+2})\MLine$\\
$n\geq 7\land \frac{n-3}{5}\leq m\leq \frac{n-4}{3}:(1-\frac{1}{4 n+6 n m-3 m-1},\frac{6 m+4}{4 n+6 n m-3 m-1})\MDots$ \\\> $(1-\frac{1}{5 n+6 n m-3 m-2},\frac{6 m+5}{5 n+6 n m-3 m-2})\MDots(1-\frac{1}{4 n+6 n m-3 m-2},\frac{6 m+4}{4 n+6 n m-3 m-2})\MLine$ \\\> $\overline{(1-\frac{1}{3 n+6 n m-3 m-1},\frac{6 m+3}{3 n+6 n m-3 m-1})}\MLine$
\\
$\mathbf{C_{10}(n,m)}:$\\
$n\geq 6\land m=\frac{n-6}{5}:(1-\frac{5}{6 n^2-9 n+8},\frac{6 n-6}{6 n^2-9 n+8})\MLine\overline{(1-\frac{5}{6 n^2-4 n+3},\frac{6 n-1}{6 n^2-4 n+3})}\MLine$ \\\> $(1-\frac{5}{6 n^2-9 n+3},\frac{6 n-6}{6 n^2-9 n+3})\MDots$\\
$n\geq 5\land \frac{n-5}{5}\leq m\leq \frac{n-5}{3}:(1-\frac{1}{6 n+6 n m-3 m-2},\frac{6 m+6}{6 n+6 n m-3 m-2})\MLine$ \\\> $\overline{(1-\frac{1}{7 n+6 n m-3 m-3},\frac{6 m+7}{7 n+6 n m-3 m-3})}\MLine(1-\frac{1}{6 n+6 n m-3 m-3},\frac{6 m+6}{6 n+6 n m-3 m-3})\MDots$ \\\> $(1-\frac{1}{5 n+6 n m-3 m-2},\frac{6 m+5}{5 n+6 n m-3 m-2})\MDots$
\\
$\mathbf{C_{11}(n,m)}:$\\
$n\geq 5\land m=\frac{n-5}{5}:(1-\frac{5}{6 n^2-8 n+10},\frac{6 n-5}{6 n^2-8 n+10})\MDots(1-\frac{5}{6 n^2-3 n+5},\frac{6 n}{6 n^2-3 n+5})\MLine$ \\\> $(1-\frac{5}{6 n^2-8 n+5},\frac{6 n-5}{6 n^2-8 n+5})\MDots$\\
$n\geq 4\land \frac{n-4}{5}\leq m\leq \frac{n-4}{3}:(1-\frac{1}{5 n+6 n m-3 m-1},\frac{6 m+5}{5 n+6 n m-3 m-1})\MDots$ \\\> $(1-\frac{1}{6 n+6 n m-3 m-2},\frac{6 m+6}{6 n+6 n m-3 m-2})\MLine(1-\frac{1}{5 n+6 n m-3 m-2},\frac{6 m+5}{5 n+6 n m-3 m-2})\MDots$ \\\> $(1-\frac{1}{4 n+6 n m-3 m-1},\frac{6 m+4}{4 n+6 n m-3 m-1})\MDots$
\\
$\mathbf{C_{12}(n,m)}:$\\
$n\geq 7\land m=\frac{n-7}{5}:(1-\frac{5}{6 n^2-10 n+6},\frac{6 n-7}{6 n^2-10 n+6})\MDots(1-\frac{5}{6 n^2-5 n+1},\frac{6 n-2}{6 n^2-5 n+1})\MDots$ \\\> $(1-\frac{5}{6 n^2-10 n+1},\frac{6 n-7}{6 n^2-10 n+1})\MDots$\\
$n\geq 6\land \frac{n-6}{5}\leq m\leq \frac{n-6}{3}:(1-\frac{1}{7 n+6 n m-3 m-3},\frac{6 m+7}{7 n+6 n m-3 m-3})\MDots$ \\\> $(1-\frac{1}{8 n+6 n m-3 m-4},\frac{6 m+8}{8 n+6 n m-3 m-4})\MDots(1-\frac{1}{7 n+6 n m-3 m-4},\frac{6 m+7}{7 n+6 n m-3 m-4})\MDots$ \\\> $(1-\frac{1}{6 n+6 n m-3 m-3},\frac{6 m+6}{6 n+6 n m-3 m-3})\MDots$
\\
$\mathbf{C_{13}(n,m)}:$\\
$n\geq 3\land m=\frac{n-3}{5}:(1-\frac{5}{6 n^2-6 n+9},\frac{6 n-3}{6 n^2-6 n+9})\MDots(1-\frac{5}{6 n^2-n+4},\frac{6 n+2}{6 n^2-n+4})\MDots$ \\\> $(1-\frac{5}{6 n^2-6 n+4},\frac{6 n-3}{6 n^2-6 n+4})\MLine$\\
$n\geq 6\land \frac{n-2}{5}\leq m\leq \frac{n-3}{3}:(1-\frac{1}{3 n+6 n m-3 m},\frac{6 m+3}{3 n+6 n m-3 m})\MDots$ \\\> $(1-\frac{1}{4 n+6 n m-3 m-1},\frac{6 m+4}{4 n+6 n m-3 m-1})\MDots(1-\frac{1}{3 n+6 n m-3 m-1},\frac{6 m+3}{3 n+6 n m-3 m-1})\MDots$ \\\> $(1-\frac{1}{2 n+6 n m-3 m},\frac{6 m+2}{2 n+6 n m-3 m})\MDots$
\\
$\mathbf{C_{14}(n,m)}:$\\
$n=2\land m=-1:\overline{(\frac{3}{4},\frac{1}{2})}\MLine(\frac{4}{5},\frac{3}{5})\MDots(\frac{2}{3},\frac{2}{3})\MLine$\\
$n\geq 6\land \frac{n-7}{5}\leq m\leq \frac{n-6}{3}:\overline{(1-\frac{1}{8 n+6 n m-3 m-3},\frac{6 m+8}{8 n+6 n m-3 m-3})}\MLine$ \\\> $(1-\frac{1}{9 n+6 n m-3 m-4},\frac{6 m+9}{9 n+6 n m-3 m-4})\MLine$\\
$n\geq 5\land m=\frac{n-5}{3}:\overline{(1-\frac{1}{2 n^2-3 n+2},\frac{2 n-2}{2 n^2-3 n+2})}\MLine(\frac{2 n^2-2 n}{2 n^2-2 n+1},\frac{2 n-1}{2 n^2-2 n+1})\MDots$ \\\> $(1-\frac{1}{2 n^2-3 n+1},\frac{2 n-2}{2 n^2-3 n+1})\MDots(1-\frac{1}{2 n^2-4 n+2},\frac{2 n-3}{2 n^2-4 n+2})\MLine$
\\
$\mathbf{C_{15}(n,m)}:$\\
$n\geq 6\land \frac{n-7}{5}\leq m\leq \frac{n-6}{3}:(1-\frac{1}{8 n+6 n m-3 m-3},\frac{6 m+8}{8 n+6 n m-3 m-3})\MLine$ \\\> $(1-\frac{1}{7 n+6 n m-3 m-3},\frac{6 m+7}{7 n+6 n m-3 m-3})\MLine$\\
$n\geq 4\land m=\frac{n-4}{3}:(1,\frac{1}{n})\MDots(1-\frac{1}{2 n^2-2 n+1},\frac{2 n-1}{2 n^2-2 n+1})\MLine\overline{(1-\frac{1}{2 n^2-3 n+2},\frac{2 n-2}{2 n^2-3 n+2})}\MLine$
\\
$\mathbf{C_{16}(n,m)}:$\\
$n\geq 7\land \frac{n-4}{5}\leq m\leq \frac{n-4}{3}:(1-\frac{1}{5 n+6 n m-3 m-2},\frac{6 m+5}{5 n+6 n m-3 m-2})\MLine$ \\\> $(1-\frac{1}{4 n+6 n m-3 m-1},\frac{6 m+4}{4 n+6 n m-3 m-1})\MLine$
\\
$\mathbf{C_{17}(n,m)}:$\\
$n\geq 5\land \frac{n-5}{5}\leq m\leq \frac{n-5}{3}:\overline{(1-\frac{1}{6 n+6 n m-3 m-2},\frac{6 m+6}{6 n+6 n m-3 m-2})}$
\\
$\mathbf{C_{18}(n,m)}:$\\
$n\geq 8\land \frac{n-4}{5}\leq m\leq \frac{n-5}{3}:\overline{(1-\frac{1}{4 n+6 n m-3 m-2},\frac{6 m+4}{4 n+6 n m-3 m-2})}$
\\
$\mathbf{C_{19}(n,m)}:$\\
$n=1\land m=0:\overline{(\frac{1}{2},\frac{3}{4})}\MLine(\frac{1}{2},1)\MLine\overline{(\frac{2}{5},\frac{4}{5})}\MLine$\\
$n\geq 3\land n \equiv 0 \bmod{3}\land m=0:(1-\frac{1}{2 n^2-3 n+2},\frac{2 n-2}{2 n^2-3 n+2})\MLine\overline{(1-\frac{1}{2 n^2-2 n+1},\frac{2 n-1}{2 n^2-2 n+1})}\MLine$\\
\> $(1-\frac{3}{4 n^2-6 n+3},\frac{4 n-3}{4 n^2-6 n+3})\MDots(1-\frac{3}{4 n^2-6 n+6},\frac{4 n-3}{4 n^2-6 n+6})\MDots$\\
$n\geq 9\land n \equiv 0 \bmod{3}\land 1\leq m\leq \frac{2 n-6}{9}:(1-\frac{3}{4 n^2-4 n+9 n m+3},\frac{4 n+9 m-3}{4 n^2-4 n+9 n m+3})\MLine$\\
\> $\overline{(1-\frac{1}{2 n^2-2 n+6 n m-3 m+1},\frac{2 n+6 m-1}{2 n^2-2 n+6 n m-3 m+1})}\MLine$\\
\> $(1-\frac{3}{4 n^2-6 n+9 n m-9 m+3},\frac{4 n+9 m-3}{4 n^2-6 n+9 n m-9 m+3})\MDots$\\
$n\geq 4\land n \equiv 1 \bmod{3}\land 0\leq m\leq \frac{2 n-8}{9}:(1-\frac{3}{4 n^2-2 n+9 n m+4},\frac{4 n+9 m-1}{4 n^2-2 n+9 n m+4})\MLine$ \\\> $\overline{(1-\frac{1}{2 n^2+6 n m-3 m},\frac{2 n+6 m+1}{2 n^2+6 n m-3 m})}\MLine(1-\frac{3}{4 n^2-4 n+9 n m-9 m},\frac{4 n+9 m-1}{4 n^2-4 n+9 n m-9 m})\MDots$\\
$n\geq 5\land n \equiv 2 \bmod{3}\land 0\leq m\leq \frac{2 n-10}{9}:\overline{(1-\frac{3}{4 n^2+3 n+9 n m+2},\frac{4 n+9 m+4}{4 n^2+3 n+9 n m+2})}\MLine$\\
\> $\overline{(1-\frac{1}{2 n^2+2 n+6 n m-3 m-1},\frac{2 n+6 m+3}{2 n^2+2 n+6 n m-3 m-1})}\MLine$\\
\> $\overline{(1-\frac{3}{4 n^2+n+9 n m-9 m-3},\frac{4 n+9 m+4}{4 n^2+n+9 n m-9 m-3})}\MLine$\\
$n\geq 6\land m=\frac{2 n-3}{9}:(1-\frac{3}{6 n^2-7 n+3},\frac{6 n-6}{6 n^2-7 n+3})\MLine\overline{(1-\frac{3}{10 n^2-14 n+6},\frac{10 n-9}{10 n^2-14 n+6})}\MLine$\\
\> $(1-\frac{3}{10 n^2-20 n+6},\frac{10 n-15}{10 n^2-20 n+6})\MDots(1-\frac{2}{4 n^2-6 n+1},\frac{4 n-4}{4 n^2-6 n+1})\MDots$\\
$n\geq 7\land m=\frac{2 n-5}{9}:(1-\frac{3}{6 n^2-7 n+4},\frac{6 n-6}{6 n^2-7 n+4})\MLine\overline{(1-\frac{3}{10 n^2-15 n+8},\frac{10 n-10}{10 n^2-15 n+8})}\MLine$\\
\> $(1-\frac{1}{2 n^2-3 n+2},\frac{2 n-2}{2 n^2-3 n+2})\MDots$\\
$n\geq 8\land m=\frac{2 n-7}{9}:\overline{(1-\frac{3}{6 n^2-4 n+2},\frac{6 n-3}{6 n^2-4 n+2})}\MLine(1-\frac{3}{10 n^2-10 n+4},\frac{10 n-5}{10 n^2-10 n+4})\MDots$\\
\> $(1-\frac{1}{2 n^2-2 n+1},\frac{2 n-1}{2 n^2-2 n+1})\MLine$
\\
\end{ntabbing}
\endgroup
\noindent
The following selection of cycles cuts out everything outside $\GC$ which proves Theorem~\ref{TMain}~(i):\\[-0.5\baselineskip]
\phantom{a}
\begin{ntabbing}
$C_{10}(n,m):$ \quad \= $n\geq 5\land \frac{n-6}{5}\leq m\leq \frac{n-5}{3}$ \,\,\= $C_{11}(n,m):$ \quad \= \kill
$C_0(1), \ldots, C_0(14), C_8(2,-1), C_2(3,-1), C_{19}(3,0)$, $C_4(6,0)$,\\
$C_1(n,m):$ \> $n\geq 2\land -1\leq m\leq \frac{n-5}{3}$, \> $C_{10}(n,m):$ \> $n\geq 5\land \frac{n-6}{5}\leq m\leq \frac{n-5}{3}$,\\
$C_2(n,m):$ \> $n\geq 4\land -1\leq m\leq \frac{n-7}{3}$, \> $C_{11}(n,m):$ \> $n\geq 4\land \frac{n-5}{5}\leq m\leq \frac{n-4}{3}$,\\
$C_3(n,m):$ \> $n\geq 5\land 0\leq m\leq \frac{n-5}{3}$, \> $C_{12}(n,m):$ \> $n\geq 6\land \frac{n-7}{5}\leq m\leq \frac{n-6}{3}$,\\
$C_4(n,m):$ \> $n\geq 7\land 0\leq m\leq \frac{n-7}{3}$, \> $C_{13}(n,m):$ \> $n\geq 3\land \frac{n-3}{5}\leq m\leq \frac{n-3}{3}$,\\
$C_5(n,m):$ \> $n\geq 2\land 0\leq m\leq \frac{n-2}{3}$, \> $C_{14}(n,m):$ \> $n\geq 6\land \frac{n-7}{5}\leq m\leq \frac{n-6}{3}$,\\
$C_6(n,m):$ \> $n\geq 4\land 0\leq m\leq \frac{n-4}{3}$, \> $C_{15}(n,m):$ \> $n\geq 6\land \frac{n-7}{5}\leq m\leq \frac{n-6}{3}$,\\
$C_7(n,m):$ \> $n\geq 5\land 0\leq m\leq \frac{n-5}{5}$, \> $C_{16}(n,m):$ \> $n\geq 7\land \frac{n-4}{5}\leq m\leq \frac{n-4}{3}$,\\
$C_8(n,m):$ \> $n\geq 3\land -1\leq m\leq \frac{n-8}{5}$, \> $C_{17}(n,m):$ \> $n\geq 5\land \frac{n-5}{5}\leq m\leq \frac{n-5}{3}$,\\
$C_9(n,m):$ \> $n\geq 4\land \frac{n-4}{5}\leq m\leq \frac{n-4}{3}$, \> $C_{18}(n,m):$ \> $n\geq 8\land \frac{n-4}{5}\leq m\leq \frac{n-5}{3}$,\\
$C_{19}(n,m):$ \> $n\geq 4\land \frac{1-n \bmod 3}{2}\leq m\leq \frac{2 n - 2 (n \bmod 3) - 5}{9},$\\
\end{ntabbing}
The figures below show a regular sector (where the polygons of the infinite families are sufficient to cut out the respective part). In the first figure it can be seen for $n = 20$ that the whole region outside $\GC$ in the sector $\frac{1}{n} < \arctan{\phi} \leq \frac{1}{n - 1}$ of the unit disk is being cut out. It can be shown by comparing the coordinates of the vertices of the polygons that this is the case for every $n \geq 7$. The subsequent figure shows the polygons moved apart in groups to illustrate how they fit together. It can be seen that the dotted lines (indicating parts of the boundary which do not belong to the corresponding polygon) of one group hit solid ones (indicating parts which to belong) of the other group and vice versa, and that single missing points are also complemented (indicated by prominent dots at the respective position). Note that the polygons from the family 19 are needed to cut out a small region remaining in the respective sector if only the families one to 18 are considered. In fact a single (but not arbitrary) polygon of family 19 would be sufficient.
\begin{figure}[H]
\includegraphics[width=10cm]{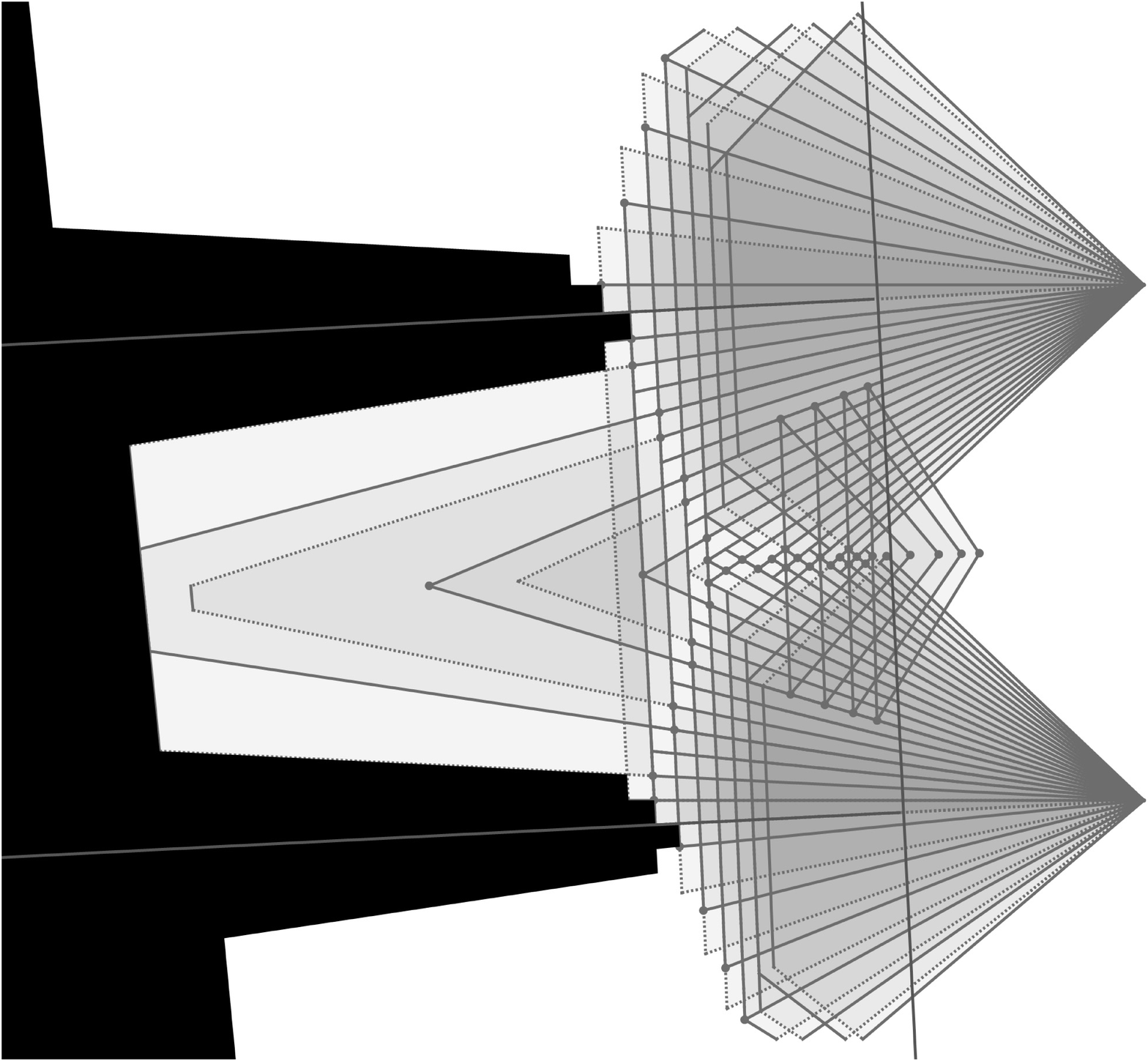}
\end{figure}
\begin{figure}[H]
\includegraphics[width=11.3cm]{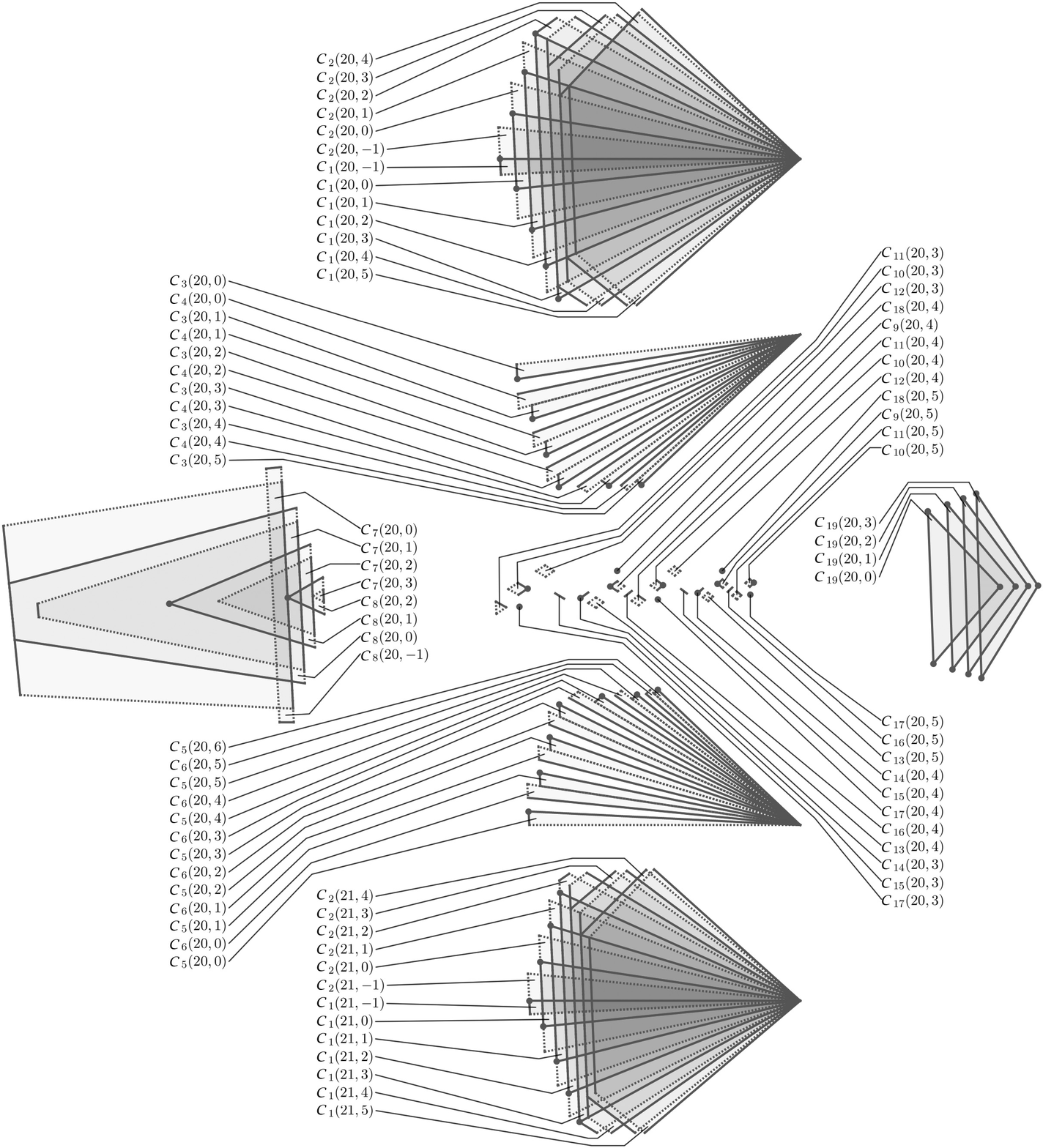}
\end{figure}
The families were found by manual search. First a list of cycles was computed the corresponding cutout polygons of which would cover the necessary parts of several successive sectors. These cycles were then grouped to families by hand.
\section{The other inclusion}
\label{SInclusion2}
By applying analogues of the two algorithms introduced in \cite{Weitzer15}, Section~3, the proof of Theorem~\ref{TMain}~(ii) could be achieved (being much more efficient it was mostly Algorithm~2 that was applied). The disk  $\set{z\in\C\mid\abs{z}\leq\frac{2047}{2048}}$, which covers the settled region, contains all pikes up to and including the 30th. Since the 8th pike is already regular and the general regular structure of the Loudspeaker is therefore verified for quite many pikes, it appears reasonable to assume that Conjecture~\ref{CMain} is in fact true. Despite best efforts a general proof could not be given by now.

Translating the two algorithms of \cite{Weitzer15} to the complex setting is straight forward. As pointed out in Section~\ref{SPreliminaries}, $\GN{d}$ is the disjoint union of those $P_\vector{r}$ the corresponding parameters $\vector{r}$ of which belong to $\GN{d}$. It is easy to see that, just as in the real case, any given convex hull $H\subseteq\G{d}$ of finitely many interior points of $\G{d}$ intersects with only finitely many of the $P_\vector{r}$ (cf. \cite{Weitzer15}, Theorem~3.2). Thus the analogue of Algorithm~1 of \cite{Weitzer15}, which essentially computes exactly those $P_\vector{r}$ which intersect with $H$, also holds for all inputs $H$. If $\vector{r}\in\D{d}$ then $V_\vector{r}$ is finite and $P_\vector{r}=\set{\vector{s} \in \R^d \mid \fa \vector{a} \in V_\vector{r}:\fa i\in\set{1,\ldots,4}: \gamma_\vector{r}^{(i)}(\vector{a})=\gamma_\vector{s}^{(i)}(\vector{a})}$ is given by a system of $16\abs{V_\vector{r}}$ linear inequalities the solution of which is a convex polyhedron. For $d=2$, $\vector{r}=(r_x,r_y)$,  and $s=(x,y)$ (we identify $\C^1\simeq\R^2$ and $\Zi^1\simeq\Z^2$) the $16$ inequalities induced by $\vector{a}=(a,b)$ are given by

\begin{tabbing}
$\fa i\in\set{1,\ldots,4}:\gamma_\vector{r}^{(i)}(\vector{a})=\gamma_\vector{s}^{(i)}(\vector{a})\Leftrightarrow$\\
$\phantom{xxx}$\=$-$\=$xa+yb-\lfloor-$\=$r_xa+r_yb\rfloor\geq0\:\land\:$\=$-$\=$xa+yb+\lfloor-$\=$r_xa-r_yb\rfloor+1>0\:\land$\kill
\>\>$xa+yb-\lfloor$\>$r_xa+r_yb\rfloor\geq0\:\land\:$\>\>$xa+yb+\lfloor-$\>$r_xa-r_yb\rfloor+1>0$\\
\>\>$xa-yb-\lfloor$\>$r_xa-r_yb\rfloor\geq0\:\land\:$\>\>$xa-yb+\lfloor-$\>$r_xa+r_yb\rfloor+1>0$\\
\>$-$\>$xa+yb-\lfloor-$\>$r_xa+r_yb\rfloor\geq0\:\land\:$\>$-$\>$xa+yb+\lfloor$\>$r_xa-r_yb\rfloor+1>0$\\
\>$-$\>$xa-yb-\lfloor-$\>$r_xa-r_yb\rfloor\geq0\:\land\:$\>$-$\>$xa-yb+\lfloor$\>$r_xa+r_yb\rfloor+1>0$\\
\>\>$xb+ya-\lfloor$\>$r_xb+r_ya\rfloor\geq0\:\land\:$\>\>$xb+ya+\lfloor-$\>$r_xb-r_ya\rfloor+1>0$\\
\>\>$xb-ya-\lfloor$\>$r_xb-r_ya\rfloor\geq0\:\land\:$\>\>$xb-ya+\lfloor-$\>$r_xb+r_ya\rfloor+1>0$\\
\>$-$\>$xb+ya-\lfloor-$\>$r_xb+r_ya\rfloor\geq0\:\land\:$\>$-$\>$xb+ya+\lfloor$\>$r_xb-r_ya\rfloor+1>0$\\
\>$-$\>$xb-ya-\lfloor-$\>$r_xb-r_ya\rfloor\geq0\:\land\:$\>$-$\>$xb-ya+\lfloor$\>$r_xb+r_ya\rfloor+1>0$.
\end{tabbing}

If $\vector{a}\neq(0,0)$ then the solution set of the system of inequalities above is the intersection of $4$ half-open squares with side lengths $1/\abs{\vector{a}}$ (if $\vector{a}=(0,0)$ then the solution set is of course equal to $\C$). The $4$ squares are arranged in a way such that the intersection of them is either a singleton, an open line segment, or a nondegenerate, open, convex polygon. Both algorithms from \cite{Weitzer15} can now be applied in the same way as for real SRS with the only difference being the systems of inequalities one has to consider.
 
\section{Consequences of the conjecture}
\label{SConsequences}
If the Loudspeaker coincides with $\GC$, its perimeter is two times the sum of all distances of successive vertices of the boundary of the intersection of $\GC$ and the first quadrant.
\begin{corollary}
If $\GN{1}=\GC$ then the perimeter of the Loudspeaker is\\
$2 \displaystyle\sum_{n = 8}^\infty \Bigg(\textstyle\frac{(n - 2)\sqrt{n^2 + 1}}{\left(n^2 - n - 1\right)\left(n^2 - 2\right)} + \frac{\sqrt{n^2 + 1}}{\left(n^2 + 1\right)\left(n^2 + n + 1\right)} + \frac{\sqrt{n^2 + 4}}{\left(n^2 + 2\right)\left(n^2 + n + 2\right)} + \frac{\sqrt{n^2 - 2 n + 2}}{n^4 - 2 n^3 + n} + \frac{\sqrt{n^2 + 1}}{n^4 + 5 n^2 + 6} + $\\
\phantom{$2 \displaystyle\sum_{n = 8}^\infty \Bigg($}$\frac{\sqrt{n^6 + n^4}}{n^6 + n^4} + \frac{(n - 1)\sqrt{n^2 + 9}}{\left(n^2 + 3\right)\left(n^2 + n + 6\right)} + \frac{\sqrt{n^2 + 2 n + 2}}{\left(n^2 + n + 1\right)\left(n^2 + n + 2\right)} + \frac{(n - 7)\sqrt{n^2 + 2 n + 5}}{\left(n^2 + n + 6\right)\left(n^2 + 2 n - 1\right)}\Bigg) - \frac{\pi^2}{3} + \frac{3845467959583\sqrt{2}}{2154669737220} + \frac{48281\sqrt{5}}{270270} + \frac{28279\sqrt{10}}{311220} + \frac{\sqrt{13}}{77} + \frac{2789\sqrt{17}}{79560} + \frac{18018457\sqrt{26}}{1214863650} + \frac{\sqrt{29}}{432} + \frac{\sqrt{34}}{322} + \frac{3453570319189\sqrt{37}}{335814194609712} + \frac{\sqrt{53}}{1479} + \frac{3\sqrt{58}}{806} + \frac{\sqrt{65}}{1653} + 6$\\[0.5\baselineskip]
which is approximately $7.0317015814551008990992430035469692210269(4)$.
\end{corollary}
The area can easily be calculated using the fact that $\GC$ is star-shaped with respect to the origin. The total area is just two times the sum of the areas of all triangles where two vertices are successive vertices of the boundary of the intersection of $\GC$ and the first quadrant and the third one is $(0, 0)$.
\begin{corollary}
If $\GN{1}=\GC$ and $\psi$ denotes the digamma function then the area of the Loudspeaker is\\
$\frac{1}{2}(\psi(9 - i) + \psi(9 + i) - \frac{1}{3}(3 - i\sqrt{3})\psi(\frac{1}{2}(17 - i\sqrt{3})) - \frac{1}{3}(3 + i\sqrt{3})\psi(\frac{1}{2}(17 + i\sqrt{3})) - \psi(9 - \sqrt{2}) - \psi(9 + \sqrt{2}) - \frac{1}{2}\psi(8 - i\sqrt{2}) - \frac{1}{2}\psi(8 + i\sqrt{2}) + \frac{1}{3}\psi(8 - i\sqrt{3}) + \frac{1}{3}\psi(8 + i\sqrt{3}) + \frac{1}{5}(5 - \sqrt{5})\psi(\frac{1}{2}(17 - \sqrt{5})) + \frac{1}{5}(5 + \sqrt{5})\psi(\frac{1}{2}(17 + \sqrt{5})) + \frac{1}{14}(7 - i\sqrt{7})\psi(\frac{1}{2}(17 - i\sqrt{7})) + \frac{1}{14}(7 + i\sqrt{7})\psi(\frac{1}{2}(17 + i\sqrt{7})) - \frac{1}{69}(23 - i\sqrt{23})\psi(\frac{1}{2}(17 - i\sqrt{23})) - \frac{1}{69}(23 + i\sqrt{23})\psi(\frac{1}{2}(17 + i\sqrt{23})) - 2\psi'(1) + \psi''(1) + \frac{6459645509579599739}{831140131659037200})$\\[0.5\baselineskip]
which is approximately $1.1616244963841538925201560564707674346082(2)$.
\end{corollary}

\section{Critical points}
\label{SCritical}
One consequence of Theorem~\ref{TMain}~(i) is the following proposition on weakly critical and critical points of $\GN{1}$. A weakly critical point is a point $\vector{r}\in\C^d$ any open neighborhood of which intersects with infinitely many cutout polyhedra. A critical point is a point $\vector{r}\in\C^d$ any open neighborhood $U$ of which satisfies that $U\setminus\GN{d}$ cannot be covered by finitely many cutout polyhedra. Both notions were first introduced in \cite{ABBPTI}, Section~7.
\begin{proposition}
$1$ and $\pm\I$ are the only weakly critical points of $\GN{1}$ and $1$ is the only critical point of $\GN{1}$.
\end{proposition}
\begin{proof}
For $n \in \N$ the line through $P_5(n)$ and $P_6(n)$ hits the origin and has a gradient of $\frac{1}{n}$. Let $r = (x, y) \in \C$ (we identify $\C$ and $\R^2$) such that $\abs{r} = 1$ and $0 < y (n - 1) \leq x$, and $z = (a, b) \in \Zi$ such that $\abs{a} + \abs{b} \leq n$ and $\Max\{\abs{a}, \abs{b}\} < n$. Then $r$ lies on the unit circle in the sector between the real axis and the line through $P_5(n - 1)$ and $P_6(n - 1)$. Then one can deduce the following cases for the product $r z = (x a - y b, x b + y a)$:\\[0.5\baselineskip]
$a > 0 \quad\land\quad b \geq 0 \quad\Rightarrow\quad a - 1 \leq x a - y b < a \quad\land\quad b < x b + y a < b + 1$\\
$a \leq 0 \quad\land\quad b > 0 \quad\Rightarrow\quad a - 1 < x a - y b < a \quad\land\quad b - 1 \leq x b + y a < b$\\
$a < 0 \quad\land\quad b \leq 0 \quad\Rightarrow\quad a < x a - y b \leq a + 1 \quad\land\quad b - 1 < x b + y a < b$\\
$a \geq 0 \quad\land\quad b < 0 \quad\Rightarrow\quad a < x a - y b < a + 1 \quad\land\quad b < x b + y a \leq b + 1$\\[0.5\baselineskip]
So the product, which is just $z$ rotated by the argument of $r$, is contained in the unit square lying next to $z$ in rotational direction. This implies a specific behavior of $\gamma_r^2(z) = \ceiling{r \floor{r z}}$ if $(a < 0 \lor b < 0) \Rightarrow \abs{a} + \abs{b} < n$:\\[0.5\baselineskip]
$a > 1 \quad\land\quad b \geq 0 \quad\Rightarrow\quad \gamma_r^2(z) = z + (-1, 1)$\\
$a = 1 \quad\land\quad b \geq 0 \quad\Rightarrow\quad \gamma_r^2(z) = z + (-1, 0)$\\
$a \leq 0 \quad\land\quad b > 1 \quad\Rightarrow\quad \gamma_r^2(z) = z + (-1, -1)$\\
$a \leq 0 \quad\land\quad b = 1 \quad\Rightarrow\quad \gamma_r^2(z) = z + (0, -1)$\\
$a < 0 \quad\land\quad b \leq 0 \quad\Rightarrow\quad \gamma_r^2(z) = z + (1, -1)$\\
$a \geq 0 \quad\land\quad b < 0 \quad\Rightarrow\quad \gamma_r^2(z) = z + (1, 1)$\\[0.5\baselineskip]
Therefore the orbits of the Gaussian integers $(n - 1, 1)$ and $(-n + 1, 0)$ both end up in $(0, 0)$ and cover the set $M_n \ce \{(a, b) \in \Zi \mid \abs{a} + \abs{b} \leq n \land ((a \leq 0 \lor b \leq 0) \Rightarrow \abs{a} + \abs{b} < n)\}$. For geometric reasons it follows that the orbit of any element of $M_n$ ends up in $(0, 0)$ even if $\abs{r} \leq 1$. In conclusion:\\[0.5\baselineskip]
$\fa n \in \N: \ex m \in \N: \fa r = (x, y) \in \C: (0 < y (n - 1) \leq x \:\land\: \abs{r} \leq 1 \Rightarrow \gamma_r^m(M_n) = \{(0, 0)\})$ \qquad ($m = n^2 - \ceiling{\frac{n}{2}}$ is a possible choice)\\[0.5\baselineskip]
Therefore all cycles of any $r$ having the properties above have empty intersection with $M_n$, which forces them to grow beyond all bounds as $n$ increases. Thus infinitely many cycles are needed to cut out, say, the set $\{P_4(n) \mid n \geq 3\}$ from the Loudspeaker which actually is being cut out entirely (Theorem~\ref{TMain}~(i)). It follows that $1$ is a critical point.\\[0.5\baselineskip]
Since for all $(a, b) \in \Zi$, $((a, b), (b, -a), (-a, -b), (-b, a))$ is a cycle of $(0, 1)$ and $((a, b), (-b, a), (-a, -b), (b, -a))$ is a cycle of $(0, -1)$, it follows that $\I$ and $-\I$ are weakly critical points. Theorem~\ref{TMain}~(i) implies that the intersection of the closure of $\GN{1}$ and the boundary of $\G{1}$ (unit circle) consists of these three points which implies that there are no other critical or weakly critical points.
\end{proof}

\end{document}